\documentclass[reqno]{amsart}%
\usepackage{enumerate}
\usepackage{amsfonts}
\usepackage{amsmath}
\usepackage{amssymb}
\usepackage{cases}
\usepackage{hyperref}
\usepackage{float}
\usepackage{graphicx}%
\newtheorem{theorem}{Theorem}[section]
\newtheorem{lemma}[theorem]{Lemma}

\theoremstyle{definition}
\newtheorem{definition}[theorem]{Definition}

\newtheorem{remark}[theorem]{Remark}

\newcommand{\norm}[1]{\left\Vert#1\right\Vert}
\newcommand{\tvh}[1]{\left<#1\right>}

\newcommand{\tap}[1]{\left\{#1\right\}}

\numberwithin{equation}{section}
\hypersetup{
colorlinks,
citecolor=blue,
filecolor=black,
linkcolor=blue,
urlcolor=magenta
}
\hypersetup{linktocpage}
\begin{document}
	\font\nho=cmr10
	\def\dive{\mathrm{div}}
	\def\cal{\mathcal}
	\def\L{\cal L}

	\def \ud{\underline }
	\def\id{{\indent }}
	\def\f{\frac}
	\def\non{{\noindent}}
	\def\le{\leqslant} 
	\def\leq{\leqslant} 
	\def\rar{\rightarrow}
	\def\Rar{\Rightarrow}
	\def\ti{\times}
	\def\i{\mathbb I}
	\def\j{\mathbb J}
	\def\si{\sigma}
	\def\Ga{\Gamma}
	\def\ga{\gamma}
	\def\ld{{\lambda}}
	\def\Si{\Psi}
	\def\f{\mathbf F}
	\def\r{\hro{R}}
	\def\e{\cal{E}}
	\def\B{\cal B}
	\def\A{\mathcal{A}}
	\def\p{\mathbb P}
	
	\def\tet{\theta}
	\def\Tet{\Theta}
	\def\hro{\mathbb}
	\def\ho{\mathcal}
	\def\P{\ho P}
	\def\E{\mathcal{E}}
	\def\n{\mathbb{N}}
	\def\M{\mathbb{M}}
	\def\dMu{\mathbf{U}}
	\def\dMcs{\mathbf{C}}
	\def\dMcu{\mathbf{C^u}}
	\def\vk{\vskip 0.2cm}
	\def\td{\Leftrightarrow}
	\def\df{\frac}
	\def\Wei{\mathrm{We}}
	\def\Rey{\mathrm{Re}}
	\def\s{\mathbb S}
	\def\l{\mathcal{L}}
	\def\C+{C_+([t_0,\infty))}
	\def\o{\cal O}
\title[Periodic solutions for Boussinesq systems]{Periodic solutions for Boussinesq systems in weak-Morrey spaces}

\author[P.T. Xuan]{Pham Truong Xuan}
\address{Pham Truong Xuan \hfill\break Faculty of Pedagogy, VNU University of
Education, Vietnam National University, 144 Xuan Thuy, Cau Giay, Hanoi, Viet Nam}
\email{phamtruongxuan.k5@gmail.com or ptxuan@vnu.edu.vn}

\author[N.T. Van]{Nguyen Thi Van}
\address{Nguyen Thi Van\hfill\break
	Department of Mathematics, Faculty of Information Technology, Thuyloi University, 175 Tay Son, Dong Da, Ha Noi, Viet Nam}
\email{van@tlu.edu.vn}

\author[T.V. Thuy]{ Tran Van Thuy}
\address{ Tran Van Thuy\hfill\break
East asia university of Technology, Polyco building, Trinh Van Bo, Nam Tu Liem, Ha Noi, Viet Nam}
\email{thuyhum@gmail.com or thuytv@eaut.edu.vn}

\begin{abstract}
We prove the existence and polynomial stability of periodic mild solutions for Boussinesq systems in critical weak-Morrey spaces for dimension $n\geqslant3$. Those systems are derived via the Boussinesq approximation and describe the movement of an incompressible viscous fluid under natural convection filling the whole space $\mathbb{R}^{n}$. Using certain dispersive and smoothing properties of heat semigroups on Morrey-Lorentz spaces as well as Yamazaki-type estimate on block spaces, we prove the existence of bounded mild solutions for the linear  {systems} corresponding to the Boussinesq systems. Then, we establish a Massera-type theorem to obtain the existence and uniqueness of periodic solutions to corresponding linear  {systems} on the half line time-axis by using a mean-ergodic method. Next, using fixed point arguments, we can pass from linear  {systems} to prove the existence uniqueness and polynomial stability of such solutions for Boussinesq systems. Finally, we apply the results to Navier-Stokes equations.

\end{abstract}
\maketitle

\tableofcontents

\font\nho=cmr10

\section{Introduction}
We are concerned with the incompressible Boussinesq system in the whole space%
\begin{equation}
\left\{
\begin{array}
[c]{rll}%
u_{t}-\Delta u+(u\cdot\nabla)u+\nabla p\!\!\!\! & =\kappa\theta g + \dive F\quad &
x\in\mathbb{R}^{n},\,t>0,\hfill\\
\operatorname{div}u\!\! & =\;0\quad & x\in\mathbb{R}^{n},\,t\geq0,\\
\theta_{t}-\Delta\theta+(u\cdot\nabla)\theta\!\! & =\; \dive f\quad & x\in
\mathbb{R}^{n},\,t>0,\\
u(x,0)\!\! & =\;u_{0}(x)\quad & x\in\mathbb{R}^{n},\\
\theta(x,0)\!\! & =\;\theta_{0}(x)\quad & x\in\mathbb{R}^{n},
\end{array}
\right.  \label{BouEq}%
\end{equation}
where $n\geqslant3$, the unknowns are $u(\cdot,t): \mathbb{R}^n\to \mathbb{R}^n $, $p(\cdot,t): \mathbb{R}^n\to \mathbb{R}$, and 
$\theta(\cdot,t):\mathbb{R}^n\to \mathbb{R}$ representing the velocity field, the pressure and the temperature of the fluid, respectively. The constant $\kappa>0$ is the volume expansion
coefficient and the field $g$ represents a generalization of the gravitational field on $\mathbb{R}^n$.  Here, we have considered the forms of the external force and the reference temperature in the system (1.1) by $\dive F$ (for $F$ is a second order tensor) and $\dive f$ (for $f$ is a vector field), respectively. These are arised from technical issues in this paper.
Note that, the divergence forms of external forces were also used in the previous works for the Navier-Stokes equations in \cite{GaSho,Ya2000} and for the Boussinesq system in \cite{HuyXuan2022}.
If we consider the zero-temperature case, i.e., $\theta=0$, then the system
(\ref{BouEq}) becomes the Navier-Stokes equations.

The  system (\ref{BouEq}) describe the movement of an incompressible viscous
fluid under the effect of natural convection filling the whole space
$\mathbb{R}^{n}$, by assuming the so-called Boussinesq approximation in which
density variations are considered only in the coupling term $\kappa\theta g$
via a buoyancy-type force (see, e.g., \cite{Chandra},\cite{Fi1969}). Moreover,
by relaxing and considering some suitable variations in that approximation,
generalizations of (\ref{BouEq}) appear in a natural way (see \cite{Ca1980}%
,\cite{Fi1969}). Thus, we also consider more general forms of $g$, including
time-dependent cases (see, e.g., \cite{Ca1980},\cite{Mo1991}), but without
losing sight of relevant cases corresponding to time-independent fields
(Remark \ref{Rem-Theo-1} $i.$).

System (\ref{BouEq}) presents the scaling
\begin{equation}
(u,\theta)\rightarrow(u_{\lambda},\theta_{\lambda}),\text{ for }\lambda>0,
\label{Scal-1}%
\end{equation}
where $(u_{\lambda},\theta_{\lambda})=\lambda(u(\lambda x,\lambda^{2}%
t),\theta(\lambda x,\lambda^{2}t))$ is the solution of (\ref{BouEq}) with the field
$g_{\lambda}(x,t)=\lambda^{2}g(\lambda x,\lambda^{2}t)$ and the initial data
$(u_{0\lambda},\theta_{0\lambda})=\lambda(u_{0}(\lambda x),\theta_{0}(\lambda
x))$ in place of $g=g(x,t)$ and $(u_{0},\theta_{0}),$ respectively. Then,
assuming the correct homogeneity for $g$ and $(u_{0},\theta_{0})$, say
$g(x,t)=\lambda^{2}g(\lambda x,\lambda^{2}t)$ and $(u_{0\lambda},\theta_{0\lambda}%
)=\lambda(u_{0}(\lambda x),\theta_{0}(\lambda x))$, it follows that
$(u_{\lambda},\theta_{\lambda})$ satisfies (\ref{BouEq}) for each $\lambda>0$
provided that $(u,\theta)$ does so. Moreover, the map (\ref{Scal-1}) induces
the following initial-data scaling
\begin{equation}
(u_{0}(x),\theta_{0}(x))\rightarrow\lambda(u_{0}(\lambda x),\theta_{0}(\lambda
x)). \label{Scal-2}%
\end{equation}
Although $u$ is a vector field and $\theta$ is a scalar field, throughout this paper, we
denote their spaces in the same way. Additionally, the divergence-free
condition is assumed for each element in the corresponding space of $u$.
Sometimes, for clarity, we may use the notation $X^{\sigma}$ to denote the
space of all $u:\mathbb{R}^{n}\rightarrow\mathbb{R}^{n}$ such that $u\in X$
and $\nabla\cdot u=0$ in the sense of tempered distributions. In view of
(\ref{Scal-2}), we say that the Banach space $X$ is critical for the Boussinesq system (\ref{BouEq}) whether its norm is invariant under \eqref{Scal-2} in the sense that $\norm{u_{0}}_{X^\sigma} \simeq \lambda\norm{u_{0\lambda}}_{X^\sigma}$ and $\norm{\theta_{0}}_{X} \simeq \lambda\norm{\theta_{0\lambda}}_{X}$, for all $u_{0}\in X^{\sigma}$ and $\theta
_{0}\in X$.  Here, the notation ``$\simeq$'' means the equivalence between two sides by multiplying a positive constant independent to $\lambda$. 

  The  Boussinesq system have been studied by several authors and  by using some approaches and techniques. To illustrate that, but without claiming
to make a complete list due to the extensive literature of (\ref{BouEq}), we
only review some of relevant works. In pioneering work, Fife and Joseph
\cite{Fi1969} provided one of the first rigorous mathematical results for the
convection problem by constructing analytic stationary solutions for
(\ref{BouEq}) with bounded fields $g$, as well as analyzing some stability and
bifurcation properties. After, Cannon and DiBenedetto \cite{Ca1980}
established the local-in-time existence in the class $L^{p}(0,T;L^{q}%
(\mathbb{R}^{n}))$ with suitable  numbers  $p,q$ and a more general coupling term
$f(\theta,x,t)$ in R.H.S. of (\ref{BouEq}), which is Lipschitz
continuous in a suitable $L^{p}(0,T;L^{q})$-sense w.r.t. $\theta$, covering in
particular $f(\theta,x,t)=\kappa\theta g$ with unbounded fields $g(x,t)$. They
also obtained global-in-time solutions under smallness conditions on the
initial data and $f$. Considering a bounded domain of $\mathbb{R}^{3}$ (see
also \cite{Mo1991}), Hishida \cite{Hi1995} obtained  the  existence and exponential
stability of global-in-time strong solutions for (\ref{BouEq}) near to the
steady state. Later, by using estimates in weak-$L^{p}$ spaces (i.e.,
$L^{p,\infty}$-$L^{q,\infty}$-type estimates) of the semigroup $e^{-tL}$
associated with the corresponding linear equation of (\ref{BouEq}),
Hishida \cite{Hi1997} showed the existence and large-time behavior of
global-in-time strong solutions in an exterior domain of $\mathbb{R}^{3}$
under smallness assumptions on the initial data $(u_{0},\theta_{0})$.
Well-posedness of time-periodic small solutions in exterior domains was proved
in \cite{Na2020} by employing frameworks in weak-$L^{p}$ spaces. The
existence and stability of global small mild solutions for (\ref{BouEq}) and the corresponding stationary Boussinesq
 system  were studied in weak-$L^{p}$ spaces by \cite{Fe2006} and
\cite{Fe2010}, respectively. The results of \cite{Fe2006} were extended for the initial data which belongs to product of Morrey spaces $\mathcal{M}^\sigma_{p,n-p}\times \mathcal{M}_{p,n-p}$ in \cite{Al2011}. A stability result in $B_{2,1}^{3/2}\times
\dot{B}_{2,1}^{-1/2},$ under small perturbations, for a class of global large
$H^{1}$- solutions was proved by \cite{Liu2014}. In the whole space
$\mathbb{R}^{3}$, Brandolese and Schonbek \cite{Br2012} obtained results on the
existence and time-decay of weak solutions for (\ref{BouEq}) with the initial data
$(u_{0},\theta_{0})\in L^{2}\times L^{2}$. Li and Wang \cite{Li-Wang2021}
analyzed (\ref{BouEq}) in the torus $\mathbb{T}^{3}$ and obtained an
ill-posedness result in $\dot{B}_{\infty,\infty}^{-1}\times\dot{B}%
_{\infty,\infty}^{-1}$ by showing the so-called norm inflation phenomena. Komo
\cite{Komo2015} analyzed (\ref{BouEq}) in general smooth domains
$\Omega\subset$ $\mathbb{R}^{3}$ and obtained uniqueness criteria for strong
solutions in the framework of Lebesgue time-spatial mixed spaces
$L^{p}(0,T;L^{q}(\Omega))$ by assuming $(u_{0},\theta_{0})\in L^{2}\times
L^{2}$ and $g\in L^{8/3}(0,T;L^{4}(\Omega))$. Considering the case of a
constant field $g$, Brandolese and He \cite{Br2020} showed the uniqueness of
mild solutions in the class $(u,\theta)\in C([0,T],L^{3}(\mathbb{R}^{3})\times
L^{1}(\mathbb{R}^{3}))$ with $\theta\in L_{loc}^{\infty}((0,T);L^{q,\infty
}(\mathbb{R}^{3}))$. They also obtained the uniqueness property for $u\in
C([0,T],L^{3}(\mathbb{R}^{3}))$ with $\theta\in L_{loc}^{\infty}%
((0,T);L^{q,\infty}(\mathbb{R}^{3})),$ $\sup_{0<t<T}t^{\frac{3}{2}%
(1-1/q)}\left\Vert \theta(\cdot,t)\right\Vert _{L^{q,\infty}}<\infty$ and
$\theta_{0}\in B_{L^{q,\infty},\infty}^{-3(1-1/q)}.$ For existence and
uniqueness results in the partial inviscid cases of (\ref{BouEq}), we   quote  \cite{Danchin2009}, \cite{Danchin2008} and their references, where the authors
explored different kinds of conditions on the initial data $(u_{0},\theta
_{0})$ involving $L^{p},$ $L^{p,\infty}$ (weak-$L^{p}$) and Besov spaces.
The unconditional uniqueness results for mild solutions of the Boussinesq system
(and also the Navier-Stokes equations) were obtained on weak-Lorentz spaces in some previous works \cite{Le2002, Ya2000} and recently in weak-Morrey spaces by Ferreira and Pham in \cite{Fe2016, FeXuan2023}.

The existence of time periodic solutions to  the Boussinesq system \eqref{BouEq} in bounded and exterior domains has been established in \cite{HuyXuan2022,Na2020,Vi2010}, respectively. On one hand, the authors in \cite{Na2020,Vi2010} considered the Boussinesq system \eqref{BouEq} in the product of weak-$L^n$ spaces, i.e., $\mathbb{L}_\sigma^{n,\infty}(\Omega)\times L^{n,\infty}(\Omega)$, where $\Omega\subset \mathbb{R}^n\, (n \geq 3)$ is an exterior domain. 
They  exploited the interpolation properties and Kato's iteration scheme to construct   periodic mild solutions on whole line time-axis $\mathbb{R}$ to \eqref{BouEq} under certain conditions on the boundedness of the temperature term and of the external force in suitable weak-Lorentz spaces. On the other hand, the authors in \cite{HuyXuan2022} considered the system \eqref{BouEq} on only a half line time-axis $\r_+$ and proved the existence, uniqueness and stability of a periodic solution to \eqref{BouEq} in this case. They used the $L^p-L^q$ smoothness, duality estimates and interpolation functors to prove the existence of bounded mild solutions for the {linear Boussinesq system} in $\mathbb{L}_\sigma^{n,\infty}(\Omega)\times L^{n,\infty}(\Omega)$. Then they invoked Massera's principle to construct an initial datum which guarantees the existence of a periodic solution for  corresponding  linear  system  to \eqref{BouEq}. Using results obtained for linear  system  and fixed point arguments, they  established the existence and stability results for \eqref{BouEq} in $\mathbb{L}_\sigma^{n,\infty}(\Omega)\times L^{n,\infty}(\Omega)$. 

 {The homogeneous weak-Morrey space $\mathcal{M}_{p,\infty,\lambda}$ (for $0<p<\infty$ and $0<\lambda<n$) belongs to a very singularity class of the family of Lorentz-Morrey spaces. The space $\mathcal{M}_{p,\infty,\lambda}$ appeared in some previous works with difference frameworks such as: potential analysis \cite{Adams,Al2011},  fluid dynamics \cite{Fe2016}, the optimal regularity of solutions to elliptic equations \cite{Phuc} and harmonic analysis \cite{Ragusa}. Besides, the space $\mathcal{M}_{p,\infty,\lambda}$ was also studied in some other works \cite{Al2013,Al2022,Hatano20,Hatano22}.
Therefore, the weak-Morrey space plays an important role in many areas of analysis.
In this paper, we continue to study fluid dynamics in this functional space.} In particular,  {inspired} from \cite{Fe2016,FeXuan2023,HuyXuan2022}, we study the existence, uniqueness and polynomial stability of periodic mild solutions for  {the} Boussinesq system \eqref{BouEq} on half line time-axis and in suitable weak-Morrey spaces. We establish the well-posedness of the system \eqref{BouEq} in $BC(\mathbb{R}_+,\, \mathcal{M}_{p,\infty,\lambda}^{\sigma}(\mathbb{R}^{n})\times\mathcal{M}_{p,\infty,\lambda}(\mathbb{R}^{n}))$ with the initial data $(u_0,\theta_0)$ which belongs to $\mathcal{M}_{p,\infty,\lambda}^{\sigma}(\mathbb{R}^{n})\times\mathcal{M}_{p,\infty,\lambda}(\mathbb{R}^{n})$, where $\lambda=n-p$. First, we use Duhamel's principle to give the integral formulation of system \eqref{BouEq} (see Subsection \ref{MatrixEqs}).
Then, we use linear and bilinear estimates \eqref{linearEs} and \eqref{bestimate} (which have been proven by using dispersive and Yamazaki-type estimates) to establish the existence of bounded mild solutions for corresponding linear  {system} to \eqref{BouEq} in Theorem \ref{Theo-uniq}. Next, we prove a Massera-type principle for the existence  {of an initial data,} which provides the periodic solution for linear equation in Theorem \ref{PeriodicLinearCasse} by  {exploiting} the limit of Ces\`aro sum, predual of weak-Morrey spaces and Banach-Alaoglu's theorem. The existence and polynomial stability of periodic solutions for  {the} Boussinesq system \eqref{BouEq} (see Theorem \ref{wellposed} and \ref{stability}) are based on a combination of fixed point arguments and some bilinear estimates \eqref{BBestimate}.
Since the weak-Morrey space $\mathcal{M}_{p,\infty,n-p}(\mathbb{R}^{n})$ are larger than weak-Lorentz space $L^{n,\infty}(\mathbb{R}^n)$,  {our stability results} extend the ones obtained in previous works \cite{Al2011,Fe2006,Fe2010,HuyXuan2022}. The new results in this paper also imply the similar ones for the Navier-Stokes equations when we consider the temperature function is zero (see for example refs. \cite{Huy2014,Ya2000} for the existence of periodic solutions for  to Navier-Stokes equations).  

This paper is organized as follows: Section \ref{S2} provides some preliminaries about Lorentz, Lorentz-Morrey spaces and the preduals of Lorentz-Morrey spaces, and we refer (\ref{BouEq}) in a suitable matrix integral form; Section \ref{S3} relies on linear estimates and the existence of periodic mild solutions for the inhomogeneous linear  {system} corresponding to \eqref{BouEq}; in Section \ref{S4}, we prove the existence and polynomial stability for periodic mild solutions of Boussinesq system \eqref{BouEq} by using results in Section \ref{S3} and fixed point arguments and we revisit the results obtained for Navier-Stokes equations. \\
{\bf Notations.}\\
$\bullet$ We denote the space of bounded and continuous functions with values in the Banach space $X$ by
$$ {BC}(\mathbb R_+,X)=\{h: \mathbb R_+ \to X|\; h \text{ is continuous and } \sup_{t\in \mathbb R_+}\|h(t)\|_X <\infty \},$$
endowed with the norm $\|h\|_{ {BC}(\mathbb R_+,X) }:=\sup\limits_{t\in \mathbb R_+}\|h(t)\|_X$.\\
$\bullet$  We denote the norm on Lorentz-Morrey space $\mathcal{M}_{p,q,\lambda}$ by $\left\Vert f\right\Vert _{p,q,\lambda}=\left\Vert f\right\Vert _{\mathcal{M}_{p,q,\lambda}}$.\\
$\bullet$  {We denote the Cartesian product spaces: $\mathcal{M}^\sigma_{p,\infty,\lambda} (\mathbb R^n) \times  \mathcal{M}_{p,\infty,\lambda}(\mathbb R^n)$ by ${\bf X}_{p,\lambda}$ and $\mathcal{M}^\sigma_{q,\infty,\mu}(\mathbb{R}^n)\times \mathcal{M}_{r,\infty,\nu}(\mathbb{R}^n)$ by ${\bf X}_{q,\mu;r,\nu}$. Throughout this paper, we use the norms $$\left\Vert 
	\begin{bmatrix}
		u\\ \theta
	\end{bmatrix}
	\right\Vert _{{\bf X}_{p,\lambda}} := \norm{u}_{ \mathcal M^\sigma_{p,\infty,\lambda}}+ \norm{\theta}_{\mathcal M_{p,\infty,\lambda}}$$
and
$$\norm{\begin{bmatrix} u\\ \theta
	\end{bmatrix}}_{{\bf X}_{q,\mu;r,\nu}}: = \norm{u}_{\mathcal M^\sigma_{q,\infty,\mu}} + \norm{\theta}_{\mathcal{M}_{r,\infty,\nu}}$$	
for the spaces ${\bf X}_{p,\lambda}$ and ${\bf X}_{q,\mu;r,\nu}$, respectively.}\\
$\bullet$ We denote $|||g|||_{\beta,b}=\sup\limits_{t>0}t^{\beta}\left\Vert g(\cdot,t)\right\Vert _{b,\infty,\lambda}$ for $\beta >0$.
%$\bullet$ We denote the Banach space $ {BC}(\mathbb R_+,  {{\bf X}_{p,\lambda}})$ by $H_{p,\infty}$.
%$\norm{h}_{H_{p,\infty}} =\sup\limits_{t>0}\norm{h(\cdot,t)}_{ {{\bf X}_{p,\lambda}}}$.	
%$\bullet$ $|||g|||_{\beta,b}=\sup\limits_{t>0}t^{\beta}\left\Vert g(\cdot,t)\right\Vert _{b,\infty,\lambda}$ with $g(\cdot,t)\in \mathcal %M_{b,\infty,\lambda}$ for all $t>0$.

\section{Preliminaries}\label{S2}
This section is devoted to some preliminaries on Lorentz spaces, Lorentz-Morrey spaces, their preduals, block spaces and basic estimates that
will be useful in the next section. For further details on those subjects, the
reader is referred to \cite{Adams,Fe2016,FeXuan2023,Graf2004}.
\subsection{Lorentz and Lorentz-Morrey spaces} 
Consider $\Omega\subset\mathbb{R}^{n}$ and the rearrangement function
$f^{\ast}(t)=\inf\left\{  s>0:m(\left\{  x\in\Omega:|f(x)|>s\right\}
)\leqslant t\right\}  ,$ for $t>0,$ where $m$ denotes  the Lebesgue measure in
$\mathbb{R}^{n}$. A measurable function $f:\Omega\rightarrow\mathbb{R}$
belongs to the Lorentz space $L^{p,q}(\Omega)$ if the norm
\begin{equation}
\left\Vert f\right\Vert _{L^{p,q}}=%
\begin{cases}
\left[  \int_{0}^{\infty}\left(  t^{\frac{1}{p}}[f^{\ast\ast}(t)]\right)
^{q}\frac{dt}{t}\right]  ^{\frac{1}{q}}, & 1<p<\infty,1\leq q<\infty,\\
\displaystyle\sup_{t>0}t^{\frac{1}{p}}[f^{\ast\ast}(t)], & 1<p\leq
\infty,\,q=\infty,
\end{cases}
\label{almost_norm}%
\end{equation}
is finite, where $f^{\ast\ast}(t)=\frac{1}{t}\int_{0}^{t}f^{\ast}(s)ds$ is the
double-rearrangement of $f$. The pair $(L^{p,q}(\Omega),\left\Vert
\cdot\right\Vert _{p,q})$ is a Banach space. Particularly, we have
$L^{p}(\Omega)=L^{p,p}(\Omega)$, and $L^{p,\infty}$ is the weak-$L^{p}$ space.

A natural generalization of Morrey spaces $\mathcal{M}_{p,\lambda}$ is the
so-called Lorentz-Morrey spaces, namely Morrey spaces based on Lorentz spaces.
For that, denote the open ball $D(a,\rho)=\left\{  x\in\mathbb{R}%
^{n};\left\vert x-a\right\vert <\rho\right\}  $ with $a\in\mathbb{R}^{n}$ and
$\rho>0$. The Morrey-Lorentz space $\mathcal{M}_{p,q,\lambda}:=\mathcal{M}%
_{p,q,\lambda}(\mathbb{R}^{n})$ is the class of all $f\in L_{loc}%
^{p,q}(\mathbb{R}^{n})$ satisfying%
\begin{equation}
\left\Vert f\right\Vert _{p,q,\lambda}:=\left\Vert f\right\Vert _{\mathcal{M}%
_{p,q,\lambda}}=\sup_{x_{0}\in\mathbb{R}^{n},\rho>0}\rho^{-\frac{\lambda}{p}%
}\left\Vert f\right\Vert _{L^{p,q}(D(x_{0},\rho))}<\infty, \label{NormLM}%
\end{equation}
where the quantity $\left\Vert \cdot\right\Vert _{p,q,\lambda}$ defines a norm
in $\mathcal{M}_{p,q,\lambda}$. The space $\mathcal{M}_{p,q,\lambda}$ endowed
with $\left\Vert \cdot\right\Vert _{p,q,\lambda}$ is a Banach space. In the
case $p=q$ and $q=\infty$, we have $\mathcal{M}_{p,p,\lambda}=\mathcal{M}%
_{p,\lambda}$ and the weak-Morrey space $\mathcal{M}_{p,\infty,\lambda}$.
Moreover, we have the scaling $\left\Vert f(cx)\right\Vert _{\mathcal{M}%
_{p,q,\lambda}}=c^{-\tau_{p,\lambda}}\left\Vert f(x)\right\Vert _{\mathcal{M}%
_{p,q,\lambda}}$, for all $c>0$, where $\tau_{p,\lambda}=\dfrac{n-\lambda}%
{p}.$

We have the inclusion (see \cite{Fe2016}):
$$\mathcal{M}_{p_2,q_2,\lambda}\hookrightarrow \mathcal{M}_{p_1,q_1,\mu},$$
for $0\leq \lambda,\mu <n,$ $1\leq p_1\leq p_2 \leq \infty, \, 1\leq q_2 \leq q_1 \leq \infty$ and $\frac{n-\mu}{p_1}=\frac{n-\lambda}{p_2}$.
Comparing with more standard critical spaces, we have the continuous inclusions (see \cite{Al2011,BeLo,KoNa1994}):
$$L^n(\mathbb{R}^n)\hookrightarrow L^{n,\infty}(\mathbb{R}^n)\hookrightarrow \mathcal{M}_{p,n-p}(\mathbb{R}^n)\hookrightarrow \mathcal{M}_{p,\infty,n-p}(\mathbb{R}^n).$$

In what follows, we recall H\"{o}lder inequality and a heat estimate in the
$\mathcal{M}_{p,q,\lambda}$-setting (see \cite{Fe2016}).

\begin{lemma}\label{heatestimate}
(i) Let $1<p_{0},p_{1},r\leqslant\infty$ and $0\leqslant\beta,\lambda
_{0},\lambda_{1}<n$ satisfy $\dfrac{1}{r}=\dfrac{1}{p_{0}}+\dfrac{1}{p_{1}}$
and $\dfrac{\beta}{r}=\dfrac{\lambda_{0}}{p_{0}}+\dfrac{\lambda_{1}}{p_{1}},$
and let $s\geqslant1$ be such that $\dfrac{1}{q_{0}}+\dfrac{1}{q_{1}}%
\geqslant\dfrac{1}{s}$. Then, the following inequality holds
\begin{equation}
\left\Vert fg\right\Vert _{r,s,\beta}\leqslant C\left\Vert f\right\Vert
_{p_{0},q_{0},\lambda_{0}}\left\Vert g\right\Vert _{p_{1},q_{1},\lambda_{1}},
\label{HolderWM}%
\end{equation}
where $C>0$ is a constant.

(ii) Let $m\in\left\{  0\right\}  \cup\mathbb{N}$, $1<p,r\leqslant\infty$,
$1\leqslant q\leqslant d\leqslant\infty$, $0\leqslant\lambda,\mu<\infty$ and
$\tau_{r,\mu}=\dfrac{n-\mu}{r}\leqslant\tau_{p,\lambda}=\dfrac{n-\lambda}{p}$.
Assume also that $\lambda=\mu$ when $p\leqslant r$. Then, we have the
estimate
\begin{equation}
\left\Vert \nabla_{x}^{m}e^{t\Delta}\varphi\right\Vert _{r,d,\mu}\leqslant
Ct^{-\frac{m}{2}-\frac{1}{2}(\tau_{p,\lambda}-\tau_{r,\mu})}\left\Vert
\varphi\right\Vert _{p,q,\lambda}, \label{disp}%
\end{equation}
for all $\varphi\in\mathcal{M}_{p,q,\lambda},$ where $C>0$ is a constant.
\end{lemma}

\bigskip
\subsection{Predual of Morrey-Lorentz spaces}
We recall the predual of Morrey-Lorentz spaces and their properties, we refer readers to ref. \cite{Fe2016, FeXuan2023} for more details. Let $1<p\leqslant\infty$, $1\leqslant q\leqslant\infty$ and
$\chi\geqslant0$ where $q=\infty$ in the case $p=\infty$. A measurable
function $b(x)$ is a $(p,q,\chi)$-block if there  exist $a\in\mathbb{R}^{n}$
and $\rho>0$ such that $\mathrm{supp}(b)\subset D(a,\rho)$ and $\rho
^{\frac{\chi}{p}}\left\Vert b\right\Vert _{L^{p,q}(D(a,\rho))}\leqslant1.$

The block space $\mathcal{P}\mathcal{D}_{p,q,\chi}:=\mathcal{P}\mathcal{D}_{p,q,\chi}(\mathbb{R}^n)$ consists of all
measurable functions $h(x)$ such that
\[
h(x)=\sum_{k=1}^{\infty}\alpha_{k}b_{k}(x),\hbox{  for a.e.,  }x\in
\mathbb{R}^{n},
\]
where $b_{k}(x)$ is a $(p,q,\chi)$-block and $\sum_{k=1}^{\infty}\alpha
_{k}<\infty$. The space $\mathcal{P}\mathcal{D}_{p,q,\chi}$ is a Banach
space endowed with the norm
\[
\left\Vert h\right\Vert _{\mathcal{P}\mathcal{D}_{p,q,\chi}}=\inf\left\{
\sum_{k=1}^{\infty}|\alpha_{k}|<\infty;\,h=\sum_{k=1}^{\infty}\alpha_{k}%
b_{k}\hbox{  where $b_k$'s are $(p,q,\chi)$-blocks}\right\}  .
\]

The relation between the preduals of Lorentz-Morrey spaces and block spaces is
given in the following lemma (see \cite[Lemma 3.1]{Fe2016}).
\begin{lemma}
\label{DualBlock} Let $1<p,q\leqslant\infty$, $0\leqslant\lambda<n$, and
$\chi\geq0$ {be} such that $\dfrac{\lambda}{p-1}=\chi$. Then, we have the duality property $\left(  \mathcal{P}\mathcal{D}_{p^{\prime},q^{\prime}%
,\chi}\right)  ^{\ast}=\mathcal{M}_{p,q,\lambda}$.
\end{lemma}

{The density property of predual Lorentz-Morrey spaces in give in the following lemma.
\begin{lemma}\label{SeparabilityWM}
	Let $1<p\leq\infty$, $1\leq q \leq \infty$ and $\chi\geqslant 0$. Then, the space $L_c^{p^\prime,q^\prime}(\mathbb R^n)$ is dense in $\cal P\cal D_{p^\prime,q^\prime,\chi} $.
As a direct consequence, the space of smooth and compact support functions	
	$C_c^\infty(\mathbb R^n)$ is dense in $\cal P\cal D_{p^\prime,q^\prime,\chi}$ and the space $\cal P\cal D_{p^\prime,q^\prime,\chi}$ is separable.
\end{lemma}
\begin{proof}
	The proof is extended from the one for the predual Morrey spaces (see for example Theorem 345 in \cite{SaFaHa2020}). Indeed, since $h\in  \cal P\cal D_{p^\prime,q^\prime,\chi}$, there exists a sequence $\{\alpha_k \}_{k=1}^\infty $ satisfying $\sum\limits_{k=1}^{\infty}|\alpha_{k}|<\infty$ and a sequence of $(p^\prime,q^\prime,\chi)$-blocks $\{b_k(x)\}_{k=1}^\infty  $ such that $h(x)=\sum\limits_{k=1}^{+\infty}\alpha_k b_k(x),\hbox{  for a.e.,  }x\in
	\mathbb{R}^{n}$. We define the function $h_N(x):=\sum\limits_{k=1}^{N}\alpha_k b_k(x),\hbox{  for a.e.,  }x\in
	\mathbb{R}^{n}$. We have that $h_N\in L_c^{p^\prime,q^\prime}(\mathbb R^n)$, and
	$\norm{h-h_N}_{\cal P\cal D_{p^\prime,q^\prime,\chi}} \leq  \sum\limits_{k=N+1}^{+\infty}|\alpha_k|\to 0$ as $N\to \infty$. Finally, since $C_c^\infty(\mathbb{R}^n)$ is dense in $L^{p',q'}_c(\mathbb{R}^n)$, it is also dense in $\cal P\cal D_{p^\prime,q^\prime,\chi}$.
\end{proof}}

The next lemma contains some interpolation properties for block spaces (see \cite[Lemma 3.2]{Fe2016}).
\begin{lemma}
\label{Interpolation} Let $1<p,p_{1},p_{0},q,q_{1},q_{0}\leqslant\infty$ and
$\chi,\chi_{1},\chi_{0}\geqslant0$ be such that $\dfrac{1}{p}%
=\dfrac{1-\eta}{p_{0}}+\dfrac{\eta}{p_{1}}$ and $\dfrac{\chi}{p}%
=\dfrac{(1-\eta)\chi_{0}}{p_{0}}+\dfrac{\eta\chi_{1}}{p_{1}}$ with
$\eta\in(0,\,1)$. If $X_{0}$ and $X_{1}$ are Banach spaces and $\mathcal{T}%
:\mathcal{P}\mathcal{D}_{p_{0},q_{0},\chi_{0}}\rightarrow X_{0}$ and
$\mathcal{T}:\mathcal{P}\mathcal{D}_{p_{1},q_{1},\chi_{1}}\rightarrow X_{1}$
are continuous linear   {operators}, then $\mathcal{T}:\mathcal{P}\mathcal{D}%
_{p,q,\chi}\rightarrow(X_{0},X_{1})_{\theta,q}$ is also linear and
continuous. As a consequence, it follows that
\[
\mathcal{P}\mathcal{D}_{p,q,\chi}\hookrightarrow\left(  \mathcal{P}%
\mathcal{D}_{p_{0},q_{0},\chi_{0}},\mathcal{P}\mathcal{D}_{p_{1}%
,q_{1},\chi_{1}}\right)  _{\theta,q}.
\]

\end{lemma}

H\"{o}lder-type inequalities work well in the context of block spaces (see the proof in \cite[Lemma 3.5]{Fe2016}).
\begin{lemma}
\label{HolderBlockSpace} Let $1<p,p_{1},p_{0}\leqslant\infty$, $1<q,q_{1}%
,q_{0}\leqslant\infty$ and $\chi,\chi_{1},\chi_{0}\geqslant0$ be such
that $\dfrac{1}{p}=\dfrac{1}{p_{0}}+\dfrac{1}{p_{1}}$ and $\dfrac{1}%
{q}\leqslant\dfrac{1}{q_{0}}+\dfrac{1}{q_{1}}$. Then, we have
\begin{equation}
\left\Vert fg\right\Vert _{\mathcal{P}\mathcal{D}_{p,q,\chi}}\leqslant
C\left\Vert f\right\Vert _{\mathcal{P}\mathcal{D}_{p_{0},q_{0},\chi_{0}}%
}\left\Vert g\right\Vert _{\mathcal{P}\mathcal{D}_{p_{1},q_{1},\chi_{1}}},
\label{HolderBlock}%
\end{equation}
where $C>0$ is a constant. Moreover, if $p=q=1$ and $\dfrac{1}{q_{0}}%
+\dfrac{1}{q_{1}}\geqslant1$, then (\ref{HolderBlock}) is   {still} valid.
\end{lemma}

Finally, we recall estimates for the heat semigroup $\left\{  e^{t\Delta
}\right\}  _{t\geqslant0}$ as well as a Yamazaki-type estimate \cite{Ya2000}
in $\mathcal{P}\mathcal{D}_{p,q,\chi}$-spaces (see \cite[Lemma 3.6 and Lemma 5.1]{Fe2016}).

\begin{lemma}
\label{disBlock}

(i) Let $1<p_{1},p_{2}\leqslant\infty$, $1\leqslant q_{1}\leqslant
q_{2}\leqslant\infty$, $\chi_{1},\chi_{2}\geqslant0$, $m\in\left\{
0\right\}  \cup\mathbb{N}$ and $\omega_{p_{2},\chi_{2}}\leqslant
\omega_{p_{1},\chi_{1}}$, where $\omega_{p_{1},\chi_{i}}=\dfrac
{n+\chi_{i}}{p_{i}}$. Moreover, suppose that $0\leqslant\dfrac{\chi_{1}%
}{p_{1}-1}=\dfrac{\chi_{2}}{p_{2}-1}<n$ when $p_{1}\leqslant p_{2}$. Assume
also that $q_{i}=\infty$ when $p_{i}=\infty$. Then, there exists a constant
$C>0$ such that 
\[
\left\Vert \nabla_{x}^{m}e^{t\Delta}\varphi\right\Vert _{\mathcal{P}%
\mathcal{D}_{p_{2},q_{2},\chi_{2}}}\leqslant Ct^{-\frac{m}{2}-\frac{1}%
{2}(\omega_{p_{1},\chi_{1}}-\omega_{p_{2},\chi_{2}})}\left\Vert
\varphi\right\Vert _{\mathcal{P}\mathcal{D}_{p_{1},q_{1},\chi_{1}}},
\]
for all $\varphi\in\mathcal{P}\mathcal{D}_{p_{1},q_{1},\chi_{1}}$.

(ii) (Yamazaki-type estimate) Let $1<p<r<\infty$ and $\chi,\alpha   {\geqslant} 0$ be
such that $\frac{\alpha}{r-1}=\frac{\chi}{p-1}<n$. Then, there exists a
constant $C>0$  {satisfying}
\[
\int_{0}^{\infty}s^{\frac{1}{2}(\omega_{p,\chi}-\omega_{r,\alpha})-\frac
{1}{2}}\left\Vert \nabla_{x}e^{s\Delta}\varphi\right\Vert _{\mathcal{P}%
\mathcal{D}_{r,1,\alpha}}ds\leqslant C\left\Vert \varphi\right\Vert
_{\mathcal{P}\mathcal{D}_{p,1,\chi}},
\]
for all $\varphi\in\mathcal{P}\mathcal{D}_{p,1,\chi}$, where $\omega
_{r,\alpha}=\dfrac{n+\alpha}{r}$ and $\omega_{p,\chi}=\dfrac{n+\chi}{p}.$
\end{lemma}

\label{S3-P}
\subsection{Boussinesq system}\label{MatrixEqs}
In order to be convenient for the reader, henceforth, one denotes $\lambda=n-p$.
We {will consider} the Boussinesq system in the
$\mathcal{M}_{p,\infty,\lambda}$-setting. Using the free-divergence condition for
$u$, system (\ref{BouEq}) can be rewritten as
\begin{equation}
\left\{
\begin{array}
[c]{rll}%
u_{t}-\Delta u+\mathbb{P}\operatorname{div}(u\otimes u)\!\! & =\kappa
\mathbb{P}(\theta g) + \mathbb{P}\dive F\quad & x\in\mathbb{R}^{n},\,t>0,\hfill\\
\operatorname{div}u\!\! & =\;0\quad & x\in\mathbb{R}^{n},\,t\geq0,\\
\theta_{t}-\Delta\theta+\operatorname{div}(\theta u)\!\! & =\; \dive f\quad &
x\in\mathbb{R}^{n},\,t>0,\\
u(x,0)\!\! & =\;u_{0}(x)\quad & x\in\mathbb{R}^{n},\\
\theta(x,0)\!\! & =\;\theta_{0}(x)\quad & x\in\mathbb{R}^{n},
\end{array}
\right.  \label{BouEq1}%
\end{equation}
where the Leray projector $\mathbb{P}$ can be expressed in terms of the Riesz
transforms $\mathcal{R}_{j}=\partial_{j}(-\Delta)^{-\frac{1}{2}},$ namely
$(\mathbb{P})_{k,j}=\delta_{kj}+\mathcal{R}_{k}\mathcal{R}_{j}$ for each
$k,j=1,2...,n$. Riesz transforms $\mathcal{R}_{j}$ are continuous from
$\mathcal{M}_{p,q,\lambda}$ to itself, for each $j=1,2...n$ (see \cite{Fe2016}).  {For the convenience in estimates in the rest of this paper, we consider the forms of the external force and the reference temperature in the system \eqref{BouEq1} by $\dive F$ (for $F$ is a second order tensor) and $\dive f$ (for $f$ is a vector field), respectively.}

We set $L:=%
\begin{bmatrix}
-\Delta & 0\\
0 & -\Delta
\end{bmatrix}
$ acting on the Cartesian product space $ {{\bf X}_{p,\lambda}}=\mathcal{M}_{p,\infty,\lambda}^{\sigma
}\times\mathcal{M}_{p,\infty,\lambda}$. Therefore, using Duhamel's principle in a matrix form, we arrive at the following
integral formulation for (\ref{BouEq1}):

\begin{equation}%
\begin{bmatrix}
u(t)\\
\theta(t)
\end{bmatrix}
=e^{-tL}%
\begin{bmatrix}
u_{0}\\
\theta_{0}%
\end{bmatrix}
+B\left(
\begin{bmatrix}
u\\
\theta
\end{bmatrix}
,%
\begin{bmatrix}
u\\
\theta
\end{bmatrix}
\right)  (t)+T_{g}(\theta)(t) + \mathcal{C}\begin{bmatrix}
F\\
f
\end{bmatrix}(t), \label{mildsol}%
\end{equation}
where the bilinear and linear operators used in the above equation are given respectively by
\begin{equation}
B\left(
\begin{bmatrix}
u\\
\theta
\end{bmatrix}
,%
\begin{bmatrix}
v\\
\xi
\end{bmatrix}
\right)  (t):=-\int_{0}^{t}\nabla_{x} \cdot e^{-(t-s)L}%
\begin{bmatrix}
\mathbb{P}(u\otimes v)\\
u\xi
\end{bmatrix}
(s)ds, \label{Bilinear}%
\end{equation}
\begin{equation}
T_{g}(\theta)(t):=\int_{0}^{t}e^{-(t-s)L}%
\begin{bmatrix}
\kappa\mathbb{P}(\theta g)\\
0
\end{bmatrix}
(s)ds \label{Couple}%
\end{equation}
and
\begin{equation}
\mathcal{C}
\begin{bmatrix}
F\\
f
\end{bmatrix} (t):=\int_{0}^{t}\nabla_{x} \cdot e^{-(t-s)L}%
\begin{bmatrix}
\mathbb{P} (F)\\
 f
\end{bmatrix}
(s)ds.
\end{equation}
\begin{remark}
\label{Rem-form-mild} It is worth noting that the integral formulation
(\ref{mildsol}) should be meant in a dual sense in the $\mathcal{M}%
_{p,\infty,\lambda}$-setting by employing the corresponding predual space
according to Lemma \ref{DualBlock}. { This means that
\begin{eqnarray*}
	\tvh{\begin{bmatrix}
			u(t)\\	\theta(t)
		\end{bmatrix},\begin{bmatrix}
			\varphi \\
			\psi 
	\end{bmatrix}}
	&=&\tvh{e^{-tL}\begin{bmatrix}
			u_0\\
			\theta_0
		\end{bmatrix},\begin{bmatrix}
			\varphi \\
			\psi 
	\end{bmatrix}}+
	\tvh{\int_0^te^{-(t-s)L}	\left( \mathcal{G}\begin{bmatrix}u\\
			\theta
		\end{bmatrix}(s) + \mathcal{F}(s) \right)ds, \begin{bmatrix}
			\varphi \\	\psi 
	\end{bmatrix}},
\end{eqnarray*}
for all $t>0$ and all $\begin{bmatrix}
	\varphi\\ \psi
\end{bmatrix} \in \mathcal{PD} _{\frac{p}{p-1},1,\frac{\lambda}{p-1}}\times \mathcal{PD} _{\frac{p}{p-1},1,\frac{\lambda}{p-1}}$. Here, we denoted that
$$\mathcal{G}\begin{bmatrix} v\\
		\eta
	\end{bmatrix}:= \begin{bmatrix}
		\mathbb{P}[-\dive (v\otimes v)+\kappa \eta g] \\
		-\dive (\eta v)
	\end{bmatrix}.$$}
\end{remark}
 {
\begin{definition}
	Let $(u_0,\theta_0)\in   {\bf X}_{p,\lambda}$. A pair functions $(u(x,t),\theta(x,t))$ satisfying
\begin{equation*}
		\lim\limits_{t\to 0^+}\left\langle \begin{bmatrix}
		u(t)\\ \theta(t)
	\end{bmatrix},\begin{bmatrix}
		\varphi\\ \psi
	\end{bmatrix}\right\rangle  =\left\langle \begin{bmatrix}
		u_0\\ \theta_0
	\end{bmatrix},\begin{bmatrix}
		\varphi\\ \psi
	\end{bmatrix}\right\rangle ,\text{ for all } \begin{bmatrix}
	\varphi\\ \psi
	\end{bmatrix} \in \mathcal{PD} _{\frac{p}{p-1},1,\frac{\lambda}{p-1}}\times \mathcal{PD} _{\frac{p}{p-1},1,\frac{\lambda}{p-1}},
\end{equation*}
is said a global mild solution for the initial value problem \eqref{BouEq1}, if $(u, \theta)$ has the integral formula \eqref{mildsol} in sense of
distribution in Remark \ref{Rem-form-mild}.
\end{definition}}

\section{Linear estimates and periodic solutions for linear  {systems}}\label{S3}
We now study the following linear system
\begin{align}\label{LinearizedSystem}
	\begin{cases}
		\dfrac{\partial }{\partial t} \begin{bmatrix}
			u\\
			\theta
		\end{bmatrix}
		+ L\begin{bmatrix}
			u\\
			\theta
		\end{bmatrix} = \mathcal{G}\begin{bmatrix}0\\
			\eta
		\end{bmatrix}+
		\mathcal{F}(t) \medskip\\
		\begin{bmatrix}
			u(0)\\
			\theta(0)
		\end{bmatrix}  = \begin{bmatrix}
			u_0\cr \theta_0
		\end{bmatrix}
		\in {{\bf X}_{p,\lambda},}
	\end{cases}
\end{align}
 {where $\cal G: BC(\r_+,{\bf X}_{p,\lambda})\to  BC(\r_+,{\bf X}_{p,\lambda})$ is given by
	$$\mathcal{G}\begin{bmatrix} v\\
		\eta
	\end{bmatrix}:= \begin{bmatrix}
		\mathbb{P}[-\dive (v\otimes v)+\kappa \eta g] \\
		-\dive (\eta v)
	\end{bmatrix} $$ with $g\in BC(\mathbb{R}_+,\cal M_{b,\infty,\lambda}) \;(\text{for } b>\frac{p}{2})$
	and
	$\mathcal{F} =\begin{bmatrix}
		\mathbb{P} \text{ div } F \\
		\text{ div } f 
	\end{bmatrix} \in BC(\mathbb{R}_+, {\bf X}_{\frac{p}{2},\lambda}). $ 
	\begin{definition}\label{MildSol1}
		Let $(u_0,\theta_0)\in {\bf X}_{p,\lambda} $. A pair functions $(u(x,t),\theta(x,t))$ is said a global mild solution for the initial value problem \eqref{LinearizedSystem}, if
	\begin{equation}\label{limito0}
		\lim\limits_{t\to 0^+}\left\langle \begin{bmatrix}
			u(t)\\ \theta(t)
		\end{bmatrix},\begin{bmatrix}
			\varphi\\ \psi
		\end{bmatrix}\right\rangle  =\left\langle \begin{bmatrix}
			u_0\\ \theta_0
		\end{bmatrix},\begin{bmatrix}
			\varphi\\ \psi
		\end{bmatrix}\right\rangle ,\text{ for all } \begin{bmatrix}
			\varphi\\ \psi
		\end{bmatrix} \in \mathcal{PD} _{\frac{p}{p-1},1,\frac{\lambda}{p-1}}\times \mathcal{PD} _{\frac{p}{p-1},1,\frac{\lambda}{p-1}},
	\end{equation}
		 and $(u, \theta)$ has the following integral formula
		\begin{equation}\label{sol1}
			\begin{bmatrix}
				u(t)\\
				\theta(t) 
			\end{bmatrix} = e^{-tL}\begin{bmatrix}
				u_0\\
				\theta_0
			\end{bmatrix} + \int_0^t e^{-(t-s)L} \left( \mathcal{G}\begin{bmatrix}0\\
				\eta
			\end{bmatrix}(s) + \mathcal{F}(s) \right) ds
		\end{equation} in sense of
		distribution for all $t > 0$, i.e.,
\begin{eqnarray*}
	\tvh{\begin{bmatrix}
			u(t)\\	\theta(t)
		\end{bmatrix},\begin{bmatrix}
			\varphi \\
			\psi 
	\end{bmatrix}}
	&=&\tvh{e^{-tL}\begin{bmatrix}
			u_0\\
			\theta_0
		\end{bmatrix},\begin{bmatrix}
			\varphi \\
			\psi 
	\end{bmatrix}}+
	\tvh{\int_0^te^{-(t-s)L}	\left( \mathcal{G}\begin{bmatrix}0\\
			\eta
		\end{bmatrix}(s) + \mathcal{F}(s) \right)ds, \begin{bmatrix}
				\varphi \\	\psi 
	\end{bmatrix}},
\end{eqnarray*}
for all $\begin{bmatrix}
	\varphi\\ \psi
\end{bmatrix} \in \mathcal{PD} _{\frac{p}{p-1},1,\frac{\lambda}{p-1}}\times \mathcal{PD} _{\frac{p}{p-1},1,\frac{\lambda}{p-1}}.$
\end{definition}}

We concern the linear operator
\begin{equation}
\mathcal{L}%
\begin{bmatrix}
f_{1}\\
f_{2}%
\end{bmatrix}
(x)=\int_{0}^{\infty}\nabla_{x} \cdot e^{-sL}%
\begin{bmatrix}
f_{1}\\
f_{2}%
\end{bmatrix}
(\cdot,s)ds \label{linearOp}.%
\end{equation}
By a duality argument and  {Assertion (ii) in Lemma \ref{disBlock}}, we are able to
estimate (\ref{linearOp}) in the weak-Morrey setting as follows (see \cite[Lemma 3.2]{FeXuan2023} for the proof):
\begin{lemma}
\label{LinearEst} Let $n\geqslant3$, $1<r<l<\infty$ and $0\leq\chi <n$
satisfy $\tau_{r,\chi}-\tau_{l,\chi}=1$, where $\tau_{l,\chi}%
=\dfrac{n-\chi}{l}$ and $\tau_{r,\chi}=\dfrac{n-\chi}{r}$. Denoting
$\left\Vert \cdot\right\Vert _{l,\infty,\chi}=\left\Vert \cdot\right\Vert
_{ {{\bf X}_{l,\chi}}}$, we have the estimate
\begin{equation}
\left\Vert \mathcal{L}%
\begin{bmatrix}
f_{1}\\
f_{2}%
\end{bmatrix}
\right\Vert _{  {{\bf X}_{l,\chi}}}\leqslant C_1\sup_{t>0}\left\Vert
\begin{bmatrix}
f_{1}\\
f_{2}%
\end{bmatrix}
(\cdot,t)\right\Vert _{ {{\bf X}_{r,\chi}}}, \label{linearEs}%
\end{equation}
for all $%
\begin{bmatrix}
f_{1}\\
f_{2}%
\end{bmatrix}
\in L^{\infty}(\mathbb R_+,  {{\bf X}_{r,\chi}}),$ where $C_1>0$ is a constant and the supremum over $t$ is
taken in the essential sense.
\end{lemma}

In order to establish the boundedness of the mild solution, we define the norm on the Cartesian product space
 { ${\bf X}_{p,\lambda}$:=$\mathcal M^\sigma_{p,\infty,\lambda}(\mathbb R^n)  \times  \mathcal M_{p,\infty,\lambda}(\mathbb R^n)$} by
$$\norm{.}_{ {{\bf X}_{p,\lambda}}} = \norm{.}_{ \mathcal M^\sigma_{p,\infty,\lambda}}+ \norm{.}_{\mathcal M_{p,\infty,\lambda}}.$$
Moreover, we denote the space of continuous and bounded functions from $\r_+$ to the space ${\bf X}_{p,\lambda}$ by
 $$H_{p,\infty}= {BC}(\r_+,  {{\bf X}_{p,\lambda}}) $$
which is a Banach space endowed with the norm  
 $$\norm{\begin{bmatrix}
 		u\\
 		\theta
 \end{bmatrix}}_{H_{p,\infty}}:=\sup\limits_{t>0}\left( \norm{u(\cdot,t)}_{p,\infty,\lambda}+\norm{\theta(\cdot,t)}_{p,\infty,\lambda} \right).$$ 

Applying inequality \eqref{linearEs}, we are able to justify the following bilinear estimate (see \cite[Lemma 3.3]{FeXuan2023}):
\begin{lemma}
\label{Bestimate} Let   $n\geqslant3$, $2<p\leqslant n$,
$\lambda=n-p$, and consider the bilinear form $B(\cdot,\cdot)$ given in
(\ref{Bilinear}). There exists a constant $K>0$ such that
\begin{equation}
\left\Vert B\left(
\begin{bmatrix}
u\\
\theta
\end{bmatrix}
,%
\begin{bmatrix}
v\\
\xi
\end{bmatrix}
\right)  \right\Vert _{H_{p,\infty}}\leqslant K\left\Vert
\begin{bmatrix}
u\\
\theta
\end{bmatrix}
\right\Vert _{H_{p,\infty}}\left\Vert
\begin{bmatrix}
v\\
\xi
\end{bmatrix}
\right\Vert_{H_{p,\infty}}, \label{bestimate}%
\end{equation}
for all $%
\begin{bmatrix}
u\\
\theta
\end{bmatrix}
,%
\begin{bmatrix}
v\\
\xi
\end{bmatrix}
\in H_{p,\infty}$.
\end{lemma}
  
 Now we state and prove the uniqueness of
mild solutions for (\ref{BouEq})\ in the $\mathcal{M}_{p,\infty,\lambda}$-setting.
\begin{theorem}\label{well-posedness}
\label{Theo-uniq} Let $n\geqslant3$, $1<p\leqslant n$, $\lambda=n-p$,  and  {$\dfrac{p}{2}<b$. Suppose that $g(\cdot,t)\in \mathcal M_{b,\infty,\lambda}$ and $|||g|||_{\beta,b}$
is bounded, where $\beta=1-\dfrac{p}{2b}.$ For $\begin{bmatrix}
	u_0\\
	\theta_0
\end{bmatrix}\in {\bf X}_{p,\lambda}$, $\eta
\in BC(\mathbb R_+; \mathcal{M}_{p,\infty,\lambda})$ and $\begin{bmatrix}
	F\\
	f
\end{bmatrix}
\in H_{\frac{p}{2},\infty}$}, the linear system (\ref{LinearizedSystem})  { have} a unique bounded mild solution   
$\begin{bmatrix}
	u\\
	\theta
\end{bmatrix}
\in H_{p,\infty}$. Furthermore,
 \begin{align}\label{estimate1}
	\left\Vert
	\begin{bmatrix} u\\
		\theta
	\end{bmatrix}
	\right\Vert_{H_{p,\infty}}
	\leqslant C\left\Vert  
	\begin{bmatrix}
		u_0\\
		\theta_0
	\end{bmatrix}\right\Vert _{ { {\bf X}_{p,\lambda}}}+  \kappa MC_2|||g|||_{\beta,b} \sup_{t>0}\norm{   {\eta(t)}}_{p,\infty,\lambda}+   C_1  \left\Vert
	\begin{bmatrix}
		F\\ f
	\end{bmatrix}
	\right\Vert_{H_{\frac{p}{2},\infty}},
\end{align}
where  $|||g|||_{\beta,b}=\sup\limits_{t>0}t^{\beta}\left\Vert g(\cdot,t)\right\Vert _{b,\infty,\lambda}$.
\end{theorem}

\begin{proof}

Following Definition \ref{MildSol1} of global mild solution, we have that the vector $\begin{bmatrix}
u\\\theta
\end{bmatrix}$ given by the integral formula \eqref{sol1} is a solution of the linear system \eqref{LinearizedSystem} if it is bounded and verifies the limit condition \eqref{limito0}. 

 {Now, we prove the boundedness for $\begin{bmatrix}
		u\\
		\theta
	\end{bmatrix}$. Indeed, we have}
\begin{align}
\left\Vert
\begin{bmatrix}
u(t)\\
\theta(t)
\end{bmatrix}
\right\Vert _{ { {\bf X}_{p,\lambda}}}  &  =\left\Vert  e^{-tL}
\begin{bmatrix}
	u_0\\
	\theta_0
\end{bmatrix}   +  \int_{0}^{t}e^{-(t-s)L}%
\begin{bmatrix}
\kappa\mathbb{P}(\eta g)\\
0
\end{bmatrix}
(s)ds    +\int_{0}^{t} e^{-(t-s)L}%
\begin{bmatrix}
	\mathbb{P} \text{ div } F\\
	\text{ div }f
\end{bmatrix}
(s)ds\right\Vert _{  { {\bf X}_{p,\lambda}}}\nonumber\\
&\leqslant \left\Vert  e^{-tL}
\begin{bmatrix}
	u_0\\
	\theta_0
\end{bmatrix}\right\Vert _{  { {\bf X}_{p,\lambda}}}+   \left\Vert \int_{0}^{t}e^{-(t-s)L}%
\begin{bmatrix}
\kappa\mathbb{P}(\eta g)\\
0
\end{bmatrix}
(s)ds\right\Vert _{  { {\bf X}_{p,\lambda}}}\cr
&+\left\Vert \int_{0}^{t}\nabla
_{x}\cdot e^{-(t-s)L}%
\begin{bmatrix}
	\mathbb{P} F\\
	f
\end{bmatrix}
(s)ds \right\Vert _{ { {\bf X}_{p,\lambda}}}\nonumber\nonumber\\
&  \leqslant C\left\Vert  
\begin{bmatrix}
	u_0\\
	\theta_0
\end{bmatrix}\right\Vert _{  {{\bf X}_{p,\lambda}}}+I_{1}(t)+I_{2}(t). \label{EST1}%
\end{align}

  { For bound} $I_1(t)$, using \eqref{disp} with $\dfrac{1}{d}=\dfrac{1}{p}%
+\dfrac{1}{b}$ and $\dfrac{p}{2b}<1$, H\"{o}lder's inequality
\eqref{HolderWM}, and the continuity of Leray projector $\mathbb{P}$, we imply that
\begin{align}
I_{1}(t)  &  \leqslant\int_{0}^{t}\left\Vert e^{-(t-s)L}%
\begin{bmatrix}
\kappa\mathbb{P}(\eta g)\\
0
\end{bmatrix}
(s)\right\Vert _{ {{\bf X}_{p,\lambda}}}ds\nonumber\\
&  \leqslant\kappa C_2\int_{0}^{t}(t-s)^{-\frac{1}{2}\left(  \frac{p}%
{d}-1\right)  }\left\Vert
\begin{bmatrix}
\eta g\\
0
\end{bmatrix}
(s)\right\Vert _{ {{\bf X}_{d,\lambda}}}ds\nonumber\\
&  \leqslant\kappa C_2\int_{0}^{t}(t-s)^{-\frac{p}{2b}}\left\Vert
g(s)\right\Vert _{b,\infty,\lambda}\left\Vert \eta(s)\right\Vert
_{p,\infty,\lambda}ds\nonumber\\
&  \leqslant\kappa C_2\sup_{t>0}t^{1-\frac{p}{2b}}\left\Vert g(t)\right\Vert
_{b,\infty,\lambda}\sup_{t>0}\left\Vert    {\eta(t)}\right\Vert _{p,\infty,\lambda
}\int_{0}^{t}(t-s)^{-\frac{p}{2b}}s^{-1+\frac{p}{2b}}ds\nonumber\\
&  \leq \kappa C_2|||g|||_{\beta,b} \sup_{t>0}\norm{ {\eta(t)}}_{p,\infty,\lambda}  \int_{0}^{1}(1-z)^{-\frac{p}{2b}}z^{-1+\frac{p}{2b}}dz\nonumber\\
&  \leq \kappa MC_2|||g|||_{\beta,b}  \sup_{t>0}\norm{   {\eta(t)}}_{p,\infty,\lambda}, \label{EST4}%
\end{align}
where $M:=\int_{0}^{1}(1-s)^{-\frac{p}{2b}}s^{-1+\frac{p}{2b}}ds<\infty$.

   { For bound} $I_2(t)$, using the Lemma \ref{LinearEst} we get that
\begin{align}\label{ESTT}
	I_{2}(t)  & = \left\Vert  \mathcal C
	\begin{bmatrix}
		F\\
		f
	\end{bmatrix}(t)
	\right\Vert _{  {{\bf X}_{p,\lambda}}} 
	  \leqslant 
	  C_1 \sup_{t>0}\left\Vert
	\begin{bmatrix}
		F(\cdot, t)\\
		f(\cdot, t)
	\end{bmatrix}
	\right\Vert _{  {{\bf X}_{\frac{p}{2},\lambda}}}=C_1  \left\Vert
	\begin{bmatrix}
		F\\ f
	\end{bmatrix}
	\right\Vert_{H_{\frac{p}{2},\infty}}.
\end{align}
Combining inequalities \eqref{EST1}, \eqref{EST4} and \eqref{ESTT}, we receive the final estimate as desired
 \begin{align*}
 	\left\Vert
 	\begin{bmatrix} u\\
 		\theta
 	\end{bmatrix}
 	\right\Vert_{H_{p,\infty}}
 	\leqslant C\left\Vert  
 	\begin{bmatrix}
 		u_0\\
 		\theta_0
 	\end{bmatrix}\right\Vert _{ {{\bf X}_{p,\lambda}}}+  \kappa MC_2|||g|||_{\beta,b} \sup_{t>0}\norm{   {\eta(t)}}_{p,\infty,\lambda}+   C_1  \left\Vert
 	\begin{bmatrix}
 		F\\ f
 	\end{bmatrix}
 	\right\Vert_{H_{\frac{p}{2},\infty}}.
 \end{align*}
 
  { It remains to clarify that
 		\begin{equation*}
 			\lim\limits_{t\to 0^+}\left\langle \begin{bmatrix}
 				u(t)\\ \theta(t)
 			\end{bmatrix},\begin{bmatrix}
 				\varphi\\ \psi
 			\end{bmatrix}\right\rangle  =\left\langle \begin{bmatrix}
 				u_0\\ \theta_0
 			\end{bmatrix},\begin{bmatrix}
 				\varphi\\ \psi
 			\end{bmatrix}\right\rangle ,\text{ for all } \begin{bmatrix}
 				\varphi\\ \psi
 			\end{bmatrix} \in \mathcal{PD} _{\frac{p}{p-1},1,\frac{\lambda}{p-1}} \times \mathcal{PD} _{\frac{p}{p-1},1,\frac{\lambda}{p-1}}.
 		\end{equation*}
 Indeed, it is clear that
 		\begin{eqnarray*}
 		\lim\limits_{t\to 0^+}	\tvh{\begin{bmatrix}
 					u(t)\\	\theta(t)
 				\end{bmatrix},\begin{bmatrix}
 					\varphi \\
 					\psi 
 			\end{bmatrix}}
 			&=&\lim\limits_{t\to 0^+} \tvh{e^{-tL}\begin{bmatrix}
 					u_0\\
 					\theta_0
 				\end{bmatrix},\begin{bmatrix}
 					\varphi \\
 					\psi 
 			\end{bmatrix}}+\lim\limits_{t\to 0^+}
 			\tvh{\int_0^te^{-(t-s)L}	\left( \mathcal{G}\begin{bmatrix}0\\
 					\eta
 				\end{bmatrix}(s) + \mathcal{F}(s) \right)ds, \begin{bmatrix}
 					\varphi \\	\psi 
 			\end{bmatrix}}\cr
 		&=&\lim\limits_{t\to 0^+} \tvh{\begin{bmatrix}
 				u_0\\
 				\theta_0
 			\end{bmatrix},e^{-tL}\begin{bmatrix}
 				\varphi \\
 				\psi 
 		\end{bmatrix}} +\begin{bmatrix}
 		0\\ 0
 	\end{bmatrix}\cr
 	&=&\lim\limits_{t\to 0^+} \tvh{\begin{bmatrix}
 	u_0\\
 	\theta_0
 	\end{bmatrix},\begin{bmatrix}
 	e^{t\bigtriangleup}\varphi \\
 	e^{t\bigtriangleup}\psi 
 \end{bmatrix}}  \cr
&=& \tvh{\begin{bmatrix}
u_0\\
\theta_0
\end{bmatrix},\begin{bmatrix}
 \varphi \\ \psi 
\end{bmatrix}}.
 		\end{eqnarray*} 
 	Our proof is complete. }
\end{proof}

 {Before investigating on the periodicity of the solution to linear system, we recall the definition of periodic function with periodicity $T$ on the half line-axis $\mathbb{R}_+$ as follows:
A function  $h \in BC(\r_+, X )$ is called periodic on time half line  if there is a number $T>0$ such that  $h(t+T)=h(t)$ for all $t>0$.}

Now we state and prove the main result of this section.
\begin{theorem}\label{PeriodicLinearCasse}
	Let $n\geqslant3$, $1<p\leqslant n$, $\lambda=n-p$,   $ {\frac{p}{2}<b}$ and $\beta=1-\frac{p}{2b}$. Assume that the external forces $\eta
	\in   {BC}(\mathbb R_+; \mathcal{M}_{p,\infty,\lambda})$,  { $\begin{bmatrix}
		F\\
		f
	\end{bmatrix}
	\in H_{\frac{p}{2},\infty}$} and $g(\cdot,t)\in \mathcal M_{b,\infty,\lambda},\;\forall t>0$ (with $|||g|||_{\beta,b}$ bounded)
  are periodic functions   {(with respect to the time) with the same periodicity $T$}. Then,  {there exists a unique initial data $(\hat{u}_0,\,\hat{\theta}_0)$ which guarantees that} the linear system \eqref{LinearizedSystem}  {with this initial data} has a unique $T$-periodic mild solution $(\hat{u},\hat{\theta}) \in H_{p,\infty}$. Moreover, $(\hat{u},\hat{\theta})$   {satisfies}
	\begin{equation}\label{estimate11}
		\norm{ \begin{bmatrix}
				\hat{u}\\
				\hat{\theta}
		\end{bmatrix}}_{H_{p,\infty}} \leq (C+1)\left( \kappa MC_2|||g|||_{\beta,b} \sup_{t>0}\norm{   {\eta(t)}}_{p,\infty,\lambda}+   C_1  \left\Vert
	\begin{bmatrix}
	F\\ f
\end{bmatrix}
\right\Vert_{H_{\frac{p}{2},\infty}}\right).
	\end{equation}
\end{theorem}
\begin{proof}
	We improve the method in \cite[Theorem 2.2]{Huy2014} to point out the existence and uniqueness of the periodic mild solution of the system \eqref{LinearizedSystem}.
	
	For each initial data $(x,\, y)\in   {{\bf X}_{p,\lambda}}$, Theorem \ref{well-posedness} shows that  there exists a unique mild solution $(u,\, \theta)\in H_{p,\infty}$ to 
	\eqref{LinearizedSystem} with the initial data $    {\begin{bmatrix}
		u_0\\
		\theta_0
		\end{bmatrix} }=\begin{bmatrix}
	x\\
	y
	\end{bmatrix}$. Therefore, we are able to define the Poincar\'e map $\P :  {{\bf X}_{p,\lambda} \to  {\bf X}_{p,\lambda}}$ as follows: For each $(x,\, y) \in  {{\bf X}_{p,\lambda}}$, we put
	\begin{equation}
		\begin{split}\label{Po}
			\P\begin{bmatrix}
				x\\
				y
			\end{bmatrix}&:=\begin{bmatrix}
				u(T)\\
				\theta(T)
			\end{bmatrix}
		\end{split}
	\end{equation}
	where $\begin{bmatrix}
			u\\
			\theta
		\end{bmatrix}\in H_{p,\infty}$
	is the unique mild solution of \eqref{LinearizedSystem} with $ {\begin{bmatrix}
			u_0\\	\theta_0
	\end{bmatrix} =\begin{bmatrix}
		x\\	y
	\end{bmatrix}.}$

According to the formula \eqref{sol1} of the solutions, we deduce that 
	\begin{eqnarray}\label{Po1}
		\P\begin{bmatrix}
			x\\
			y
		\end{bmatrix} = \begin{bmatrix}
			u(T)\\
			\theta(T)
		\end{bmatrix} =e^{-TL }\begin{bmatrix}
			x\\
			y
		\end{bmatrix}+\int_0^T  e^{-(T-s)L}    {\left( \mathcal{G} \begin{bmatrix}
			0\\	\eta
		\end{bmatrix}(s) + \mathcal{F}(s)\right)}ds
	\end{eqnarray}
	with $\begin{bmatrix}
		u\\
		\theta
	\end{bmatrix}$  as in \eqref{Po}. 
	Thereby, from $T$-periodicity of $F$, $f$, $\eta $ and $ g$, it follows that
	\begin{eqnarray*}
		\begin{bmatrix}
			u((k+1)T)\\
			\theta((k+1)T)
		\end{bmatrix} &=& e^{-(k+1)TL}   {\begin{bmatrix}
			u_0\\
			\theta_0
			\end{bmatrix} }+\int\limits^{(k+1)T}_{0} e^{-((k+1)T-s)L} \left( \mathcal{G} \begin{bmatrix}
			0\\
			\eta 
		\end{bmatrix}(s) + \mathcal{F}(s)\right)ds \cr
		&=& e^{-(k+1)TL}  {\begin{bmatrix}
			u_0\\
			\theta_0
			\end{bmatrix} }+ \int\limits^{kT}_{0}  e^{-((k+1)T-s)L} \left( \mathcal{G} \begin{bmatrix}
			0\\
			\eta
		\end{bmatrix}(s) + \mathcal{F}(s)\right) ds \cr
		&&+ \int\limits^{(k+1)T}_{kT}  e^{-((k+1)T-s)L} \left( \mathcal{G} \begin{bmatrix}
			0\\
			\eta
		\end{bmatrix}(s) + \mathcal{F}(s)\right) ds\cr
		&=& e^{-TL}e^{-kTL}   {\begin{bmatrix}
			u_0\\
			\theta_0
			\end{bmatrix} }+ \int\limits^{kT}_{0}e^{-TL}e^{-(kT-s)L} \left( \mathcal{G} \begin{bmatrix}
			0\\
			\eta
		\end{bmatrix}(s) + \mathcal{F}(s)\right)ds \cr
		&&+\int\limits^{T}_{0}e^{-(T-s)L} \left( \mathcal{G} \begin{bmatrix}
			0\\
			\eta
		\end{bmatrix}(s) + \mathcal{F}(s)\right)ds\cr
		&=& e^{-TL}\begin{bmatrix}
			u(kT)\\
			\theta(kT)
		\end{bmatrix} + \int\limits^{T}_{0}e^{-(T-s)L} \left( \mathcal{G} \begin{bmatrix}
			0\\
			\eta
		\end{bmatrix} (s)+ \mathcal{F}(s)\right)ds\hbox{ for all }k\in \n.
	\end{eqnarray*}
	
	This implies  that  
	$\P^k\begin{bmatrix}
		x\\
		y
	\end{bmatrix}= \begin{bmatrix}
		u(kT)\\
		\theta(kT)
	\end{bmatrix}\hbox{ for all }k\in \n.$
	Hence,  $\left\{\P^k\begin{bmatrix}
		x\\
		y
	\end{bmatrix}\right\}_{k\in \n}$ is a  bounded sequence  in $   {\bf X}_{p,\lambda}$.
	
	For each $n\in \n$ we define the Ces\`aro sume  $\mathbf{P}_n$ as follows
	\begin{equation}\label{cesa}
		\mathbf{P}_n:=\frac{1}{n}\sum_{k=1}^n\P^k :
		    {\bf X}_{p,\lambda}\to 	   {\bf X}_{p,\lambda}.
	\end{equation}
	To prove the existence of initial data $\begin{bmatrix}
		\hat	u_0\\\hat \theta_0
	\end{bmatrix}=\begin{bmatrix}
	\hat	x\\\hat y
	\end{bmatrix}$ satisfying Theorem \ref{PeriodicLinearCasse}, we consider a specific initial data $\begin{bmatrix}
		x\\y
	\end{bmatrix}= \begin{bmatrix}
		0\\0
	\end{bmatrix}$    $\in  {\bf X}_{p,\lambda}$. By using the inequality \eqref{estimate1}, we obtain
	\begin{equation}\label{boundp}
		\sup_{k\in\n}\norm{\P^k\begin{bmatrix}
				0\\0
		\end{bmatrix}}_{  {{\bf X}_{p,\lambda}}}\le  {\kappa MC_2} |||g|||_{\beta,b} \sup_{t>0}\norm{   {\eta(t)}}_{p,\infty,\lambda} + M\norm{\begin{bmatrix}
				F\\f
		\end{bmatrix}}_{H_{\frac{p}{2},\infty}}.
	\end{equation}
	The boundedness of  $\left\{\P^k\begin{bmatrix}
		0\\
		0
	\end{bmatrix}\right\}_{k\in \n}$ in  
	$   {{\bf X}_{p,\lambda}}$ implies  that the sequence
	$$\left\{\mathbf{P}_n\begin{bmatrix}
		0\\
		0
	\end{bmatrix}\right\}_{n\in \n}=\tap{\frac{1}{n}\sum_{k=1}^n\P^k\begin{bmatrix}
			0\\
			0
	\end{bmatrix}}_{n\in \n}$$
	is clearly a bounded sequence in  
	$ {{\bf X}_{p,\lambda}}$. By \eqref{boundp} we obtain  
	\begin{equation}\label{boundP}
		\sup_{n\in\n}\norm{\mathbf{P}_n\begin{bmatrix}
				0\\
				0
		\end{bmatrix}}_{  {{\bf X}_{p,\lambda}}}\le  {\kappa MC_2} |||g|||_{\beta,b} \sup_{t>0}\norm{    {\eta(t)}}_{p,\infty,\lambda} + M\norm{\begin{bmatrix}
				F\\
				f
		\end{bmatrix}}_{H_{\frac{p}{2},\infty}}.
	\end{equation}  
	Since $ {{\bf X}_{p,\lambda}=} \mathcal{M}^\sigma_{p,\infty,\lambda} \times \mathcal{M}_{p,\infty,\lambda}$ has a separable pre-dual $\mathcal{P}\mathcal{D}_{\frac{p}{p-1},1,\frac{\lambda}{p-1}} \times \mathcal{P}\mathcal{D}_{\frac{p}{p-1},1,\frac{\lambda}{p-1}}$, by Banach-Alaoglu's Theorem there exists a subsequence $\left\{\mathbf{P}_{n_k}\begin{bmatrix}
		0\\
		0
	\end{bmatrix}\right\}$ of 
	$\left\{\mathbf{P}_{n}\begin{bmatrix}
		0\\
		0
	\end{bmatrix}\right\}$ such that
	\begin{equation}\label{suse9}
		\left\{\mathbf{P}_{n_k}\begin{bmatrix}
			0\\
			0
		\end{bmatrix}\right\}\ {\xrightarrow{weak\hbox{-}^*}}{}\  \begin{bmatrix}
			\hat{x}\\
			\hat{y}
		\end{bmatrix}\in  {{\bf X}_{p,\lambda}}
	\end{equation}
	with
	\begin{equation}\label{suse}
		\norm{\begin{bmatrix}
				\hat{x}\\
				\hat{y}
		\end{bmatrix}}_{  {{\bf X}_{p,\lambda}}}\le  {\kappa MC_2} |||g|||_{\beta,b} \sup_{t>0}\norm{   {\eta(t)}}_{p,\infty,\lambda} +  M\norm{\begin{bmatrix}
				F\\
				f
		\end{bmatrix}}_{H_{\frac{p}{2},\infty}}.
	\end{equation}
	By simple computations and using formula \eqref{cesa} we receive 
	$$\P\mathbf{P}_n\begin{bmatrix}
		0\\
		0
	\end{bmatrix}-\mathbf{P}_n\begin{bmatrix}
		0\\
		0
	\end{bmatrix}=\frac{1}{n}\left(\P^{n+1}\begin{bmatrix}
		0\\
		0
	\end{bmatrix}-\P\begin{bmatrix}
		0\\
		0
	\end{bmatrix} \right).$$
	Since the sequence  $\left\{\P^{n+1}\begin{bmatrix}
		0\\
		0
	\end{bmatrix}\right\}_{n\in\n}$ is bounded in $   {{\bf X}_{p,\lambda}}$,
	we get
	$$\lim_{n\to\infty}\left(\P\mathbf{P}_n\begin{bmatrix}
		0\\
		0
	\end{bmatrix}-\mathbf{P}_n\begin{bmatrix}
		0\\
		0
	\end{bmatrix} \right)=\lim_{n\to\infty}\frac{1}{n}\left(\P^{n+1}\begin{bmatrix}
		0\\
		0
	\end{bmatrix}-\P\begin{bmatrix}
		0\\
		0
	\end{bmatrix}\right)=0 \hbox{ strongly in }   {{\bf X}_{p,\lambda}}.$$
	As a consequence, for the subsequence  $\left\{\mathbf{P}_{n_k}\begin{bmatrix}
		0\\
		0
	\end{bmatrix}\right\}$, we have 
	$\P\mathbf{P}_{n_k}\begin{bmatrix}
		0\\
		0
	\end{bmatrix}-\mathbf{P}_{n_k}\begin{bmatrix}
		0\\
		0
	\end{bmatrix} \longrightarrow 0$ strongly in $   {{\bf X}_{p,\lambda}}$
	This limit  together with \eqref{suse9}  implies that 
	\begin{equation}\label{suse2}
		\P\mathbf{P}_{n_k}\begin{bmatrix}
			0\\
			0
		\end{bmatrix}\ {\xrightarrow{weak\hbox{-}^*}}{}\ \begin{bmatrix}
			\hat{x}\\
			\hat{y}
		\end{bmatrix}\in  {{\bf X}_{p,\lambda}}.
	\end{equation}
	  {Now, we claim that} $\P\begin{bmatrix}
		\hat{x}\\
		\hat{y}
	\end{bmatrix}=\begin{bmatrix}
		\hat{x}\\
		\hat{y}
	\end{bmatrix}$. To do this,
	using the formula \eqref{Po1} and denoting by $\tvh{\cdot,\cdot}$ the dual pair between $  {{\bf X}_{p,\lambda}}$ and $\mathcal{P}\mathcal{D}_{\frac{p}{p-1},1,\frac{\lambda}{p-1}} \times \mathcal{P}\mathcal{D}_{\frac{p}{p-1},1,\frac{\lambda}{p-1}}$. 
	Then, since $ {e^{-TL}}$ leaves $\mathcal{P}\mathcal{D}_{\frac{p}{p-1},1,\frac{\lambda}{p-1}} \times \mathcal{P}\mathcal{D}_{\frac{p}{p-1},1,\frac{\lambda}{p-1}}$ invariant, for all $\begin{bmatrix}
		\varphi\\
		\psi
	\end{bmatrix}\in \mathcal{P}\mathcal{D}_{\frac{p}{p-1},1,\frac{\lambda}{p-1}} \times \mathcal{P}\mathcal{D}_{\frac{p}{p-1},1,\frac{\lambda}{p-1}}$, we see that
	\begin{eqnarray}
		\tvh{\P\mathbf{P}_{n_k}\begin{bmatrix}
				0\\
				0
			\end{bmatrix},\begin{bmatrix}
				\varphi\\
				\psi
		\end{bmatrix}}
		&=&\tvh{e^{-TL}\mathbf{P}_{n_k}\begin{bmatrix}
				0\\
				0
			\end{bmatrix},\begin{bmatrix}
				\varphi\\
				\psi
		\end{bmatrix}}+
		\tvh{\int_0^Te^{-(T-s)L} { \left( \mathcal{G}\begin{bmatrix}0\\
				\hat\theta
			\end{bmatrix}(s) + \mathcal{F}(s)\right)   ds, \begin{bmatrix}
					\varphi\\	\psi
		\end{bmatrix}}}\cr
		&=&\tvh{\mathbf{P}_{n_k}\begin{bmatrix}
				0\\
				0
			\end{bmatrix},   {e^{-TL}}\begin{bmatrix}
				\varphi\\
				\psi
		\end{bmatrix}}+\tvh{\int_0^Te^{-(T-s)L}{ \left( \mathcal{G}\begin{bmatrix}0\\
			\hat\theta
		\end{bmatrix}(s) + \mathcal{F}(s)\right)   ds, \begin{bmatrix}
			\varphi\\	\psi
	\end{bmatrix}}}\cr
		&\xrightarrow{n_k\to\infty}&\tvh{\begin{bmatrix}
				\hat{x}\\
				\hat{y}
			\end{bmatrix},   {e^{-TL}}\begin{bmatrix}
				\varphi\\
				\psi
		\end{bmatrix}}+\tvh{\int_0^Te^{-(T-s)L}{ \left( \mathcal{G}\begin{bmatrix}0\\
			\hat\theta
		\end{bmatrix}(s) + \mathcal{F}(s)\right)   ds, \begin{bmatrix}
			\varphi\\	\psi
	\end{bmatrix}}}\cr
		&=&\tvh{  {e^{-TL}}\begin{bmatrix}
				\hat{x}\\
				\hat{y}
			\end{bmatrix},\begin{bmatrix}
				\varphi\\
				\psi
		\end{bmatrix}}+\tvh{\int_0^Te^{-(T-s)L}{ \left( \mathcal{G}\begin{bmatrix}0\\
			\hat\theta
		\end{bmatrix}(s) + \mathcal{F}(s)\right)   ds, \begin{bmatrix}
			\varphi\\	\psi
	\end{bmatrix}}}\cr
		&=&\tvh{\P\begin{bmatrix}
				\hat{x}\\
				\hat{y}
			\end{bmatrix},\begin{bmatrix}
				\varphi\\
				\psi
		\end{bmatrix}}.
	\end{eqnarray} 
	This leads to the fact that 
	\begin{equation}\label{suse3}
		\P\mathbf{P}_{n_k}\begin{bmatrix}
			0\\
			0
		\end{bmatrix}\ {\xrightarrow{weak\hbox{-}^*}}{}\  \P\begin{bmatrix}
			\hat{x}\\
			\hat{y}
		\end{bmatrix}\in 
		 { {\bf X}_{p,\lambda}}.
	\end{equation} 
	It now follows from \eqref{suse2} and \eqref{suse3} that 
	\begin{equation}\label{suse4}\P\begin{bmatrix}
			\hat{x}\\
			\hat{y}
		\end{bmatrix}=\begin{bmatrix}
			\hat{x}\\
			\hat{y}
		\end{bmatrix}.
	\end{equation}
	Taking now the element $\begin{bmatrix}
		\hat{x}\\
		\hat{y}
	\end{bmatrix}\in    { {\bf X}_{p,\lambda}}$ as an initial condition,  by Theorem \ref{well-posedness} there exists a unique mild solution $\begin{bmatrix}
		\hat{u}(\cdot)\\
		\hat{\theta}(\cdot)
	\end{bmatrix}\in {H_{p,\infty}}$ satisfying $   {\begin{bmatrix}
		\hat u_0\\
		\hat \theta_0
		\end{bmatrix}}=\begin{bmatrix}
		\hat{x}\\
		\hat{y}
	\end{bmatrix}$.
	From the definition of Poincar\'e map $\P$ we arrive at $   {\begin{bmatrix}
		\hat u_0\\
		\hat \theta_0
		\end{bmatrix}}=\begin{bmatrix}
		\hat{u}(T)\\
		\hat{\theta}(T)
	\end{bmatrix}$. 
Therefore,	the solution $\begin{bmatrix}
		\hat{u}(t)\\
		\hat{\theta}(t)
	\end{bmatrix}$ is periodic function with the  {periodicity} $T$.
	The inequality \eqref{estimate11} now follows from inequalities \eqref{estimate1} and \eqref{suse}.
	
We now proceed with proving the uniqueness of the periodic mild solution. Indeed, let  {$\begin{bmatrix}
		\hat u^1\\
		\hat \theta^1
		\end{bmatrix}$ and
		$ \begin{bmatrix}
		\hat u^2\\
		\hat \theta^2
		\end{bmatrix}$} be two $T$-periodic mild solutions to the system \eqref{LinearizedSystem} which belong to  $ {H_{p,\infty}}$. Then, setting
	$\begin{bmatrix}
		v\\
		\eta
	\end{bmatrix}= { \begin{bmatrix}
		\hat{u}^1\\
		\hat{\theta}^1
	\end{bmatrix}-\begin{bmatrix}
		\hat{u}^2\\
		\hat{\theta}^2
		\end{bmatrix}}$ we get that $\begin{bmatrix}
		v\\
		\eta
	\end{bmatrix}$ is $T$-periodic and, by the formula \eqref{sol1},
	\begin{equation}\label{Stokes2}
		\begin{bmatrix}
			v(t)\\
			\eta(t)
		\end{bmatrix}=e^{-tL }\left(    {\begin{bmatrix}
			\hat u_0^1\\
			\hat \theta_0^1
		\end{bmatrix} - \begin{bmatrix}
			\hat u_0^2\\
			\hat \theta_0^2
			\end{bmatrix}}\right)\hbox{ for }t>0.
	\end{equation}
 {	Then, by utilizing the dispersive estimates in Assertion (ii) of Lemma \ref{disp} we have that for $t>0$, 
	\begin{eqnarray*}
		\norm{\begin{bmatrix}
				v(t)\\
				\eta(t)
		\end{bmatrix}}_{ { {\bf X}_{p,\lambda}}}&=&
		\norm{e^{-tL }\left( \begin{bmatrix}
				\hat u_0^1\\
				\hat \theta_0^1
			\end{bmatrix} - \begin{bmatrix}
				\hat u_0^2\\
				\hat \theta_0^2
			\end{bmatrix} \right)}_{ { {\bf X}_{p,\lambda}}} =	
\norm{\begin{bmatrix}
e^{t\Delta}(\hat{u}_0^1-\hat{u}_0^2)\\e^{t\Delta}(\hat{\theta}_0^1 - \hat{\theta}_0^2)
\end{bmatrix}}_{{\bf X}_{p,\lambda}}
\cr
&=&  \norm{e^{t\Delta}(\hat{u}_0^1-\hat{u}_0^2)}_{p,\infty,\lambda} + \norm{e^{t\Delta}(\hat{\theta}_0^1 - \hat{\theta}_0^2)}_{p,\infty,\lambda}\cr
&\leq& t^{-\frac{1}{2}\left(\frac{p}{d}-1 \right)}\norm{(\hat{u}_0^1-\hat{u}_0^2)}_{d,\infty,\lambda} + t^{-\frac{1}{2}\left(\frac{p}{d}-1 \right)}\norm{(\hat{\theta}_0^1 - \hat{\theta}_0^2)}_{d,\infty,\lambda} \cr
&\leq &t^{-\frac{1}{2}\left(\frac{p}{d}-1 \right)  }	\norm{ \begin{bmatrix}
				\hat u_0^1\\
				\hat \theta_0^1
			\end{bmatrix} - \begin{bmatrix}
				\hat u_0^2\\
				\hat \theta_0^2
		\end{bmatrix}}_{ { {\bf X}_{d,\lambda}}},
	\end{eqnarray*}
where $\dfrac{1}{d}=\dfrac{1}{p}+\dfrac{1}{b}$.
	Combining this inequality with the fact that $t^{-\frac{1}{2}\left(\frac{p}{d}-1 \right)}$ tends to zero as $t$ tends to infinity (due to $\dfrac{p}{d}-1>0$), we deduce that }	
	\begin{equation}\label{Stab}
		\lim_{t\to \infty}\norm{\begin{bmatrix}
				v(t)\\
				\eta(t)
		\end{bmatrix}}_{  { {\bf X}_{p,\lambda}}}=0.
	\end{equation}
	This fact together with the periodicity of $v$ and $\eta$ implies that $\begin{bmatrix}
		v(t)\\
		\eta(t)
	\end{bmatrix}=0$ for all $t\geqslant 0$. This yields    $ { \begin{bmatrix}
		\hat{u}^1\\
		\hat{\theta}^1
	\end{bmatrix}=\begin{bmatrix}
		\hat{u}^2\\
		\hat{\theta}^2
		\end{bmatrix}}$ and the uniqueness {of such solution} holds.
\end{proof}

\section{Periodic solutions for Boussinesq systems and their polynomial stability}\label{S4}
\subsection{Existence of periodic solutions}  {In this section, we investigate the well-posedness of  periodic solution of the system (\ref{BouEq1}) on the space $H_{p,\infty}=  {BC}(\r_+,   { {\bf X}_{p,\lambda}}) $, where $\lambda=n-p$.
For this purpose, we will use the well-posed results for inhomogeneous linear system obtained in Theorem \ref{PeriodicLinearCasse}.}
\begin{theorem}\label{wellposed}(Well-posedness of periodic mild solutions). Let $n\geqslant 3$, $2<p\leq n$ and $  {\dfrac{p}{2}<b}$. Assume that the external forces $g(\cdot,t)\in \mathcal M_{b,\infty,\lambda}$ for all $t>0$ and
  { $\begin{bmatrix}F\\  f
 	\end{bmatrix} \in H_{\frac{p}{2},\infty}$ are periodic functions   {with respect to the time and they have the same periodicity $T$}}. The well-posedness of periodic mild solutions of Boussinesq system is stated as: if the norms $|||g|||_{\beta,b}=\sup\limits_{t>0}t^{1-\frac{p}{2b}}\norm{g(\cdot,t)}_{b,\infty,\lambda}$ and  { $\norm{\begin{bmatrix}F\\ f
 	\end{bmatrix}}_{H_{\frac{p}{2},\infty}}$} are small enough,   { then there exists a unique initial data $\begin{bmatrix}
 	\hat u_0\\
 	\hat\theta_0
 \end{bmatrix}$ which guarantees that the system \eqref{BouEq1} with this initial data has a unique $T$-periodic mild solution $\begin{bmatrix}
		\hat u\\
		\hat\theta
	\end{bmatrix}$ in a small ball of $H_{p,\infty}$.}
\end{theorem}
\begin{proof}
	Firstly, the ball including the set of all $T$-periodic functions $\begin{bmatrix}v\\
	\eta
\end{bmatrix} \in H_{p,\infty}$ with radius $\rho$ and centered at $\begin{bmatrix}0\\
	0
\end{bmatrix}$ is denoted by $B_\rho^T$.
 { For each $\begin{bmatrix}v\\
		\eta
	\end{bmatrix} \in  B_\rho^T$, by using Theorem \ref{PeriodicLinearCasse}, there exists a unique initial data $\begin{bmatrix}
		u_0\\
		\theta_0
	\end{bmatrix} \in {\bf X}_{p,\lambda}$ such that with this initial data the following linear system has a unique $T$-periodic solution $\begin{bmatrix}
	u\\
	\theta
	\end{bmatrix}\in {H_{p,\infty}}$:
\begin{align}\label{Linonho}
	\begin{cases}
		\dfrac{\partial }{\partial t} \begin{bmatrix}
			u\\
			\theta
		\end{bmatrix}
		+ L\begin{bmatrix}
			u\\
			\theta
		\end{bmatrix} = \mathcal{G}\begin{bmatrix}0\\
			\eta
		\end{bmatrix}+
		\widetilde{ \mathcal{F}}(t)  \medskip\\
		\begin{bmatrix}
			u(0)\\
			\theta(0)
		\end{bmatrix}  = \begin{bmatrix}
			u_0\cr \theta_0
		\end{bmatrix}
		\in {\bf X}_{p,\lambda},
	\end{cases}
\end{align}
  here $$\widetilde{ \mathcal{F}}(t)=\begin{bmatrix}
		\mathbb{P} \dive[-  v\otimes v+F] \\
		\dive (-\eta v+ f) 
	\end{bmatrix}(t).$$
Moreover, the $T$-periodic solution $\begin{bmatrix}
	u\\
	\theta
\end{bmatrix}$ satisfies}
\begin{eqnarray}\label{2.19n}
	\begin{bmatrix}
		u(t)\\
		\theta(t) 
	\end{bmatrix} &=& { e^{-tL}\begin{bmatrix}
		u_0\\
		\theta_0
	\end{bmatrix} + \int_0^t e^{-(t-s)L} \left[\mathcal{G}\begin{bmatrix}
		0\\
		\eta
	\end{bmatrix}(s) + \widetilde{ \mathcal{F}}(s)\right] ds.} \cr
	&=& { e^{-tL}\begin{bmatrix}
		u_0\\
		\theta_0
	\end{bmatrix} + \int_0^t e^{-(t-s)L}\left[\begin{bmatrix}
		\mathbb{P}[\kappa \eta g] \\
		0
	\end{bmatrix}(s)  +\begin{bmatrix}
	\mathbb{P} \dive[-  v\otimes v+F] \\
		\dive (-\eta v+ f) 
	\end{bmatrix}(s)\right]ds  }
\cr
&=&  { e^{-tL}\begin{bmatrix}
	u_0\\
	\theta_0
\end{bmatrix} + \int_0^t e^{-(t-s)L}\left[\mathcal{G}\begin{bmatrix}
	v \\
	\eta 
\end{bmatrix}(s) +\mathcal{F}(s)\right]ds, }\cr
&=& {   e^{-tL}\begin{bmatrix}
	u_0\\
	\theta_0
\end{bmatrix} + B\left(\begin{bmatrix}
	v\\
	\eta
\end{bmatrix}, \begin{bmatrix}
	v\\
	\eta
\end{bmatrix} \right)(t) + T_g(\eta)(t)+\cal C\begin{bmatrix}
	F\\
	f
\end{bmatrix} (t),}
\end{eqnarray}	
 { here $B$ and $T_g$ given by \eqref{Bilinear} and \eqref{Couple} respectively,
	and
	\begin{equation}\label{limit}
		\lim\limits_{t\to 0^+}\left\langle \begin{bmatrix}
			u(t)\\ \theta(t)
		\end{bmatrix},\begin{bmatrix}
			\varphi\\ \psi
		\end{bmatrix}\right\rangle  =\left\langle \begin{bmatrix}
			u_0\\ \theta_0
		\end{bmatrix},\begin{bmatrix}
			\varphi\\ \psi
		\end{bmatrix}\right\rangle ,\text{ for all } \begin{bmatrix}
			\varphi\\ \psi
		\end{bmatrix} \in \mathcal{PD} _{\frac{p}{p-1},1,\frac{\lambda}{p-1}}\times \mathcal{PD} _{\frac{p}{p-1},1,\frac{\lambda}{p-1}}.
	\end{equation}
}
 {Therefore, we can define the transformation  $\Phi: B_{\rho}^T \to B_{\rho}^T$ as follows
	\begin{equation}
		\Phi \begin{bmatrix}v\\
			\eta
		\end{bmatrix}(t): = \begin{bmatrix}
		u(t)\\
		\theta(t)
		\end{bmatrix},
	\end{equation} here $\begin{bmatrix}
	u\\
	\theta
\end{bmatrix}$ is given by \eqref{2.19n} (which is $T$-periodic solution of the system \eqref{Linonho}).}

Now, we prove that the map $\Phi$ is a contraction operator. Indeed, utilizing Lemma \ref{Bestimate} and Theorem \ref{well-posedness}, we recognize the fact that for all $\begin{bmatrix}v\\
		\eta
\end{bmatrix}$ belongs to the ball   $B_\rho^T$, then
	\begin{eqnarray}\label{Phibound}
		\norm{\Phi \begin{bmatrix}v\\
				\eta
		\end{bmatrix}(t)}_{  {{\bf X}_{p,\lambda}}}&\leq& \norm{e^{-tL}\begin{bmatrix}
				u_0\\
				\theta_0
		\end{bmatrix}}_{ {{\bf X}_{p,\lambda}}} + \norm{B\left(\begin{bmatrix}
				v\\
				\eta
			\end{bmatrix}, \begin{bmatrix}
				v\\
				\eta
			\end{bmatrix} \right)(t)}_{  {{\bf X}_{p,\lambda}}} + \norm{T_g(\eta)(t)}_{   {{\bf X}_{p,\lambda}}}+ \norm{  {\cal C}\begin{bmatrix}
			F\\
			f
			\end{bmatrix} (t)}_{ {{\bf X}_{p,\lambda}}}\cr
		&\leq& C\norm{\begin{bmatrix}
				u_0\\
				\theta_0
		\end{bmatrix}}_{ {{\bf X}_{p,\lambda}}} + K\norm{\begin{bmatrix}
				v\\
				\eta
		\end{bmatrix}}^2_{H_{p,\infty}} + \kappa MC_2|||g|||_{\beta,b} \sup_{t>0}\norm{   {\eta(t)}}_{p,\infty,\lambda}+C_1\norm{\begin{bmatrix}
				F\\f 
		\end{bmatrix}}_{H_{\frac{p}{2},\infty}}\cr
		&\leq& C\norm{\begin{bmatrix}
				u_0\\
				\theta_0
		\end{bmatrix}}_{ {{\bf X}_{p,\lambda}}} + K\rho^2 + \kappa MC_2\rho  |||g|||_{\beta,b} +C_1\norm{\begin{bmatrix}
		F\\f 
	\end{bmatrix}}_{H_{\frac{p}{2},\infty}}.
	\end{eqnarray}
 {By using inequality \eqref{suse} in the proof of Theorem \ref{PeriodicLinearCasse}, we have
\begin{equation}\label{susenonlinear}
	\norm{\begin{bmatrix}
			u_0\\
			\theta_0
	\end{bmatrix}}_{{\bf X}_{p,\lambda}}\leq 	  \kappa MC_2 |||g|||_{\beta,b} \sup_{t>0}\norm{ \eta(t)}_{p,\infty,\lambda} +  C_1\norm{\begin{bmatrix}
F\\
f
\end{bmatrix}}_{H_{\frac{p}{2},\infty}}.
\end{equation}  
Plugging \eqref{susenonlinear} into \eqref{Phibound}, we obtain that
 \begin{eqnarray}\label{Phibound2}
 	\norm{\Phi \begin{bmatrix}v\\
 			\eta
 	\end{bmatrix}}_{ {{\bf X}_{p,\lambda}}} (t)&\leq&  \kappa C MC_2 |||g|||_{\beta,b} \sup_{t>0}\norm{ \eta(t)}_{p,\infty,\lambda} +  C C_1\norm{\begin{bmatrix}
 	F\\
 	f
 \end{bmatrix}}_{H_{\frac{p}{2},\infty}}\cr
&+&   K\rho^2 + \kappa MC_2\rho  |||g|||_{\beta,b} +C_1\norm{\begin{bmatrix}
 			F\\f 
 	\end{bmatrix}}_{H_{\frac{p}{2},\infty}}\cr
 &\leq&  K\rho^2 + \rho\kappa MC_2(C+1) |||g|||_{\beta,b} +C_1(C+1)\norm{\begin{bmatrix}
 		F\\f 
 \end{bmatrix}}_{H_{\frac{p}{2},\infty}}<\rho
 	\end{eqnarray} 
	if $\rho$,  $\norm{\begin{bmatrix}
	F\\f 
\end{bmatrix}}_{H_{\frac{p}{2},\infty}}$ and $|||g|||_{\beta,b}$ are small enough. Therefore, we obtain that $\Phi(B_\rho^T)\subset B_\rho^T$.}

 {	Moreover, for all $\begin{bmatrix}v\\
		\eta
\end{bmatrix}$ and $\begin{bmatrix}\omega\\
	\xi
\end{bmatrix}  \in B_\rho^T$, by the same way as above we can define
	\begin{eqnarray}
  \Phi \begin{bmatrix}\omega\\
  	\xi
  \end{bmatrix} (t)&=&e^{-tL}\begin{bmatrix}
			\omega_0\\
			\xi_0
			\end{bmatrix} + B\left(\begin{bmatrix}
				\omega\\
				\xi
			\end{bmatrix}, \begin{bmatrix}
				\omega\\
				\xi
			\end{bmatrix} \right)(t) + T_g(\xi)(t)+\cal C\begin{bmatrix}
				F\\
				f
			\end{bmatrix} (t),
	\end{eqnarray}
	where $\begin{bmatrix}
		\omega_0\\
		\xi_0
	\end{bmatrix}$ is the initial data which guarantees that $\Phi \begin{bmatrix}\omega\\
	\xi
	\end{bmatrix}$ is a unique $T$-periodic solution of the system \eqref{Linonho} with the right hand side of the first equation is
$$ \mathcal{G}\begin{bmatrix}0\\
			\xi
		\end{bmatrix}(t)+
		\begin{bmatrix}
	\mathbb{P} \dive[-  \omega\otimes \omega+F] \\
	\dive (-\xi \omega+ f) 
\end{bmatrix}(t).$$
}	
Hence, it implies that 

	\begin{eqnarray}\label{1contraction}
		\norm{\Phi \begin{bmatrix} \omega\\ \xi
			\end{bmatrix} (t)- \Phi \begin{bmatrix}v\\
				\eta
		\end{bmatrix}(t)}_{{\bf X}_{p,\lambda}} &\leq&\norm{e^{-tL}\left[ \begin{bmatrix}
		\omega_0\\ \xi_0
	\end{bmatrix}- \begin{bmatrix}
	v_0\\
	\eta_0
\end{bmatrix}\right]}_{{\bf X}_{p,\lambda}}+ \norm{B\left(\begin{bmatrix}
				\omega\\ \xi
			\end{bmatrix}, \begin{bmatrix}
				\omega\\ \xi
			\end{bmatrix} \right)(t) - B\left(\begin{bmatrix}
				v\\
				\eta
			\end{bmatrix}, \begin{bmatrix}
				v\\
				\eta
			\end{bmatrix} \right)(t)}_{ {\bf X}_{p,\lambda}} \cr
&+& \norm{T_g(\xi)(t) - T_g(\eta)(t)}_{{\bf X}_{p,\lambda}}\cr
		&\leq&C\norm{\begin{bmatrix}
				\omega_0\\ \xi_0
			\end{bmatrix}- \begin{bmatrix}
				v_0\\
				\eta_0
			\end{bmatrix}}_{{\bf X}_{p,\lambda}}+ K\left( \norm{\begin{bmatrix}
				\omega\\ \xi
		\end{bmatrix}}_{H_{p,\infty}} + \norm{\begin{bmatrix}
				v\\
				\eta
		\end{bmatrix}}_{H_{p,\infty}} \right)\norm{\begin{bmatrix}
				\omega-v\\
				\xi-\eta
		\end{bmatrix}}_{H_{p,\infty}} \cr
		&&+ \kappa MC_2 |||g|||_{\beta,b}  \norm{\begin{bmatrix}
				\omega-v\\
				\xi-\eta
		\end{bmatrix}}_{H_{p,\infty}}\cr
		&\leq&C\norm{\begin{bmatrix}
				\omega_0\\ \xi_0
			\end{bmatrix}- \begin{bmatrix}
				v_0\\
				\eta_0
		\end{bmatrix}}_{{\bf X}_{p,\lambda}}+ \left(2\rho K+ \kappa MC_2 |||g|||_{\beta,b}\right) \norm{\begin{bmatrix}
				\omega-v\\
				\xi-\eta
		\end{bmatrix}}_{H_{p,\infty}}.
	\end{eqnarray}
By the same way to prove  \eqref{suse}, we can show that
 \begin{equation*}
	\norm{\begin{bmatrix}
			\omega_0\\ \xi_0
		\end{bmatrix}- \begin{bmatrix}
			v_0\\
			\eta_0
	\end{bmatrix}}_{{\bf X}_{p,\lambda}}\leq \kappa MC_2 |||g|||_{\beta,b} \sup_{t>0}\norm{\xi(t)- \eta(t)}_{p,\infty,\lambda}
\end{equation*}
Plugging this into \eqref{1contraction}, we have
\begin{eqnarray*}
	\norm{\Phi \begin{bmatrix}\omega\\ \xi
		\end{bmatrix} (t)- \Phi \begin{bmatrix}v\\
			\eta
		\end{bmatrix}(t)}_{{\bf X}_{p,\lambda}} &\leq&
	\kappa C MC_2 |||g|||_{\beta,b} \sup_{t>0}\norm{\xi(t)- \eta(t)}_{p,\infty,\lambda}\cr
&&+ \left(2\rho K+ \kappa MC_2 |||g|||_{\beta,b}\right) \norm{\begin{bmatrix}
			\omega-v\\
			\xi-\eta
	\end{bmatrix}}_{H_{p,\infty}}\cr
 &\leq&
   \Big(2\rho K+ \kappa MC_2(C +1) |||g|||_{\beta,b}\Big) \norm{\begin{bmatrix}
		\omega-v\\
		\xi-\eta
\end{bmatrix}}_{H_{p,\infty}}
\end{eqnarray*}
 {due to $\norm{\begin{bmatrix}
			\omega\\ \xi
	\end{bmatrix}}_{H_{p,\infty}}<\rho $ and $ \norm{\begin{bmatrix}
			v\\	\eta
	\end{bmatrix}}_{H_{p,\infty}}<\rho$.}
This shows that the transformation $\Phi$ is a contraction operator if $2\rho K+ \kappa MC_2(C +1) |||g|||_{\beta,b}<1$ provided by $\rho$ and $|||g|||_{\beta,b}$ small enough.  Based on the fixed point argument there exists a unique fixed point $\begin{bmatrix}
		\hat u\\
		\hat\theta
	\end{bmatrix}$ of $\Phi$ in $B_\rho^T$.  {Combining this with the limit \eqref{limit},  we obtain that $\begin{bmatrix}
		\hat u\\
		\hat\theta
	\end{bmatrix}$ is the unique solution of the system \eqref{BouEq1} (with the initial data $\begin{bmatrix}
	\hat u_0\\
	\hat\theta_0
	\end{bmatrix}$) in $B_\rho^T$ with $\rho$ small enough.}
\end{proof}

\subsection{Asymptotic stability}
In order to establish the asymptotic stability of the mild solution obtained in Theorem \ref{wellposed}, 
 { we denote the Cartesian product space $\mathcal{M}^\sigma_{q,\infty,\mu}\times \mathcal{M}_{r,\infty,\nu}$ by ${\bf X}_{q,\mu;r,\nu}$ and we define the norm on ${\bf X}_{q,\mu;r,\nu}$ by
	$$\norm{\begin{bmatrix} u\\ \theta
	\end{bmatrix}}_{{\bf X}_{q,\mu;r,\nu}}: = \norm{u}_{q,\infty,\mu} + \norm{\theta}_{r,\infty,\nu}.$$
We will establish the asymptotic stability of the mild solution obtained in Theorem \ref{wellposed} in the space ${\bf X}_{q,\mu;r,\nu}$. For this purpose, we define the following time-dependent functional space
$$H_{q,r,\infty}: = \left\{ \begin{bmatrix}
	u\cr \theta
\end{bmatrix} \in H_{p,\infty} : \sup_{t>0}\norm{\begin{bmatrix}
		t^{\alpha/2}u\cr t^{\gamma/2}\theta
\end{bmatrix}}_{{\bf X}_{q,\lambda;r,\lambda}} <\infty  \right\}$$
endowed with the norm
$$\norm{ \begin{bmatrix} u\\ 
		\theta
\end{bmatrix} }_{H_{q,r,\infty}}: = \norm{\begin{bmatrix}
		u\cr \theta
\end{bmatrix}}_{H_{p,\infty}} + \sup_{t>0}\norm{\begin{bmatrix}
		t^{\alpha/2}u\cr t^{\gamma/2}\theta
\end{bmatrix}}_{ {\bf X}_{q,\lambda;r,\lambda}},$$
where $\alpha = 1-\dfrac{p}{q}$ and $\gamma = 1-\dfrac{p}{r}$ with $1<p<q \leq r<\infty$.}

Now, we extend the bilinear estimate obtained in Theorem \ref{Bestimate} to $H_{q,r,\infty}$ as follows
\begin{theorem}(Bilinear estimate in $H_{q,r,\infty}$)
	Let $n\geqslant 3$ and $1<p<q \leq r<\infty$. Let $B(\cdot,\cdot)$ be the bilinear form \eqref{Bilinear}. There exists a constant $K > 0$ such that
	\begin{equation}\label{BBestimate}
		\norm{B\left( \begin{bmatrix}
				u\\
				\theta
			\end{bmatrix}, \begin{bmatrix}
				v\\
				\xi
			\end{bmatrix} \right)}_{H_{q,r,\infty}} \leq K \norm{\begin{bmatrix}
				u\\
				\theta
		\end{bmatrix}}_{H_{q,r,\infty}}\norm{\begin{bmatrix}
				v\\
				\xi
		\end{bmatrix}}_{H_{q,r,\infty}}
	\end{equation} 
	for all $\begin{bmatrix}
		u\\
		\theta
	\end{bmatrix}, \begin{bmatrix}
		v\\
		\xi
	\end{bmatrix} \in H_{q,r,\infty}$.
\end{theorem}
\begin{proof}
	We have that
	\begin{eqnarray}
		B\left(\begin{bmatrix}
			u\\
			\theta
		\end{bmatrix}, \begin{bmatrix}
			v\\
			\xi
		\end{bmatrix} \right)(t) &:=& -\int_0^t \nabla_x e^{-(t-s)L}\begin{bmatrix}\mathbb{P}(u\otimes v)\\
			u\xi
		\end{bmatrix}(s)ds \cr
		&=& -\int_0^t \begin{bmatrix}
			\nabla_x e^{(t-s)\Delta}\mathbb{P}(u\otimes v)\\
			\nabla_x e^{(t-s)\Delta}(u \xi)
		\end{bmatrix}(s)ds\cr
		&=& \begin{bmatrix}
			B_1(u,v)(t)\\
			B_2(u,\xi)(t)
		\end{bmatrix},
	\end{eqnarray}
	where
	\begin{equation}
		B_1(u,v)(t):= -\int_0^t \nabla_x e^{(t-s)\Delta}\mathbb{P}(u\otimes v)(s)ds,\,\, B_2(u,\xi)(t):= \int_0^t \nabla_xe^{(t-s)\Delta}(u \xi)(s)ds.
	\end{equation}
	Let $\dfrac{1}{d}=\dfrac{1}{q}+\dfrac{1}{r}$. By using Assertion (ii) in Lemma  \ref{heatestimate} and the continuous of Leray projector $\mathbb{P}$, we estimate that
	\begin{eqnarray}\label{ine1}
		&&\norm{\begin{bmatrix}
				t^{\frac{\alpha}{2}}B_1(u,v)(t)\\
				t^{\frac{\gamma}{2}}B_2(u,\xi)(t)
		\end{bmatrix}}_{ {{\bf X}_{q,\lambda;r,\lambda}}} \leq C\int_0^t \norm{\begin{bmatrix}
				t^{\frac{\alpha}{2}}(t-s)^{-\frac{1}{2}-\frac{p}{2q}}(u\otimes v)\\
				t^{\frac{\gamma}{2}}(t-s)^{-\frac{1}{2}-\frac{p}{2q}}(u \xi)
			\end{bmatrix}(s)}_{  {{\bf X}_{\frac{q}{2},\lambda;d,\lambda}}}  ds\cr
		&&\leq C\int_0^t \left( t^{\frac{\alpha}{2}}(t-s)^{-\frac{1}{2}-\frac{p}{2q}}\norm{(u\otimes v)(s)}_{\frac{q}{2},\infty,   {\lambda}} +  t^{\frac{\gamma}{2}}(t-s)^{-\frac{1}{2}-\frac{p}{2q}}\norm{(u \xi)(s)}_{d,\infty,   {\lambda}} \right) ds\cr
		&&:= C(J_1 + J_2),
	\end{eqnarray}
	where
	\begin{eqnarray*}
		&&J_1= \int_0^t t^{\frac{\alpha}{2}}(t-s)^{-\frac{1}{2}-\frac{p}{2q}}\norm{(u\otimes v)(s)}_{\frac{q}{2},\infty,   {\lambda}} ds\cr
		&&J_2= \int_0^t t^{\frac{\gamma}{2}}(t-s)^{-\frac{1}{2}-\frac{p}{2q}}\norm{(u \xi)(s)}_{d,\infty,   {\lambda}} ds.
	\end{eqnarray*}
	
From H\"older inequality in Lemma \ref{HolderWM} we can estimate $J_1$ and $J_2$ as follows
	\begin{eqnarray}\label{ine2}
		J_1 &\leq&  \int_0^t t^{\frac{\alpha}{2}}(t-s)^{-\frac{1}{2}-\frac{p}{2q}}\norm{u(s)}_{q,\infty,  {\lambda}}\norm{v(s)}_{q,\infty,  {\lambda}} ds\cr
		&\leq& t^{\frac{\alpha}{2}}\int_0^t(t-s)^{-\frac{1}{2}-\frac{p}{2q}}s^{-\alpha} ds \sup_{t>0}t^{\frac{\alpha}{2}}\norm{u(t)}_{q,\infty,  {\lambda}}\sup_{t>0}t^{\frac{\alpha}{2}}\norm{v(t)}_{q,\infty,  {\lambda}} \cr
	&\leq& t^{\frac{1}{2}-\frac{p}{2q}-\frac{\alpha}{2}}\int_0^1 (1-s)^{-\frac{1}{2}-\frac{p}{2q}}s^{-\alpha} ds \sup_{t>0}t^{\frac{\alpha}{2}}\norm{u(t)}_{q,\infty,  {\lambda}}\sup_{t>0}t^{\frac{\alpha}{2}}\norm{v(t)}_{q,\infty, {\lambda}} \cr
		&\leq& \int_0^1 (1-s)^{-\frac{1}{2}-\frac{p}{2q}}s^{-\alpha} ds \sup_{t>0}t^{\frac{\alpha}{2}}\norm{u(t)}_{q,\infty,  {\lambda}}\sup_{t>0}t^{\frac{\alpha}{2}}\norm{v(t)}_{q,\infty,  {\lambda}} \cr
		&\leq& \left( \int_0^{1/2} (1-s)^{-\frac{1}{2}-\frac{p}{2q}}s^{-\alpha} ds + \int_{1/2}^1 (1-s)^{-\frac{1}{2}-\frac{p}{2q}}s^{-\alpha} ds  \right)\cr
		&&\times \sup_{t>0} \norm{\begin{bmatrix}
				t^{\frac{\alpha}{2}}u(t)\\
				t^{\frac{\gamma}{2}}\theta(t)
		\end{bmatrix}}_{  {{\bf X}_{q,\lambda;r,\lambda}}} 
		\sup_{t>0} \norm{\begin{bmatrix}
				t^{\frac{\alpha}{2}}v(t)\\
				t^{\frac{\gamma}{2}}\xi(t)
		\end{bmatrix}}_{ {{\bf X}_{q,\lambda;r,\lambda}}}\cr
		&\leq& \left( 2^{\frac{1}{2}+\frac{p}{2q}}\int_0^{1/2}s^{-\alpha} ds + 2^{1-\frac{p}{q}}\int_{1/2}^1 (1-s)^{-\frac{1}{2}-\frac{p}{2q}} ds  \right)\cr
		&&\times \sup_{t>0} \norm{\begin{bmatrix}
				t^{\frac{\alpha}{2}}u(t)\\
				t^{\frac{\gamma}{2}}\theta(t)
		\end{bmatrix}}_{ {{\bf X}_{q,\lambda;r,\lambda}}} 
		\sup_{t>0} \norm{\begin{bmatrix}
				t^{\frac{\alpha}{2}}v(t)\\
				t^{\frac{\gamma}{2}}\xi(t)
		\end{bmatrix}}_{  {{\bf X}_{q,\lambda;r,\lambda}}}\cr
		&\leq& \left( \frac{q 2^{\frac{1}{2}-\frac{p}{2q}}}{p} + \frac{2^{\frac{1}{2}-\frac{p}{2q}}}{\frac{1}{2}-\frac{p}{2q}} \right)\sup_{t>0} \norm{\begin{bmatrix}
				t^{\frac{\alpha}{2}}u(t)\\
				t^{\frac{\gamma}{2}}\theta(t)
		\end{bmatrix}}_{  {{\bf X}_{q,\lambda;r,\lambda}}} 
		\sup_{t>0} \norm{\begin{bmatrix}
				t^{\frac{\alpha}{2}}v(t)\\
				t^{\frac{\gamma}{2}}\xi(t)
		\end{bmatrix}}_{  {{\bf X}_{q,\lambda;r,\lambda}}}\cr
		&\leq& C_1 \sup_{t>0} \norm{\begin{bmatrix}
				t^{\frac{\alpha}{2}}u(t)\\
				t^{\frac{\gamma}{2}}\theta(t)
		\end{bmatrix}}_{  {{\bf X}_{q,\lambda;r,\lambda}}} 
		\sup_{t>0} \norm{\begin{bmatrix}
				t^{\frac{\alpha}{2}}v(t)\\
				t^{\frac{\gamma}{2}}\xi(t)
		\end{bmatrix}}_{  {{\bf X}_{q,\lambda;r,\lambda}}},
	\end{eqnarray}
	where
	$$C_1 = \frac{q 2^{\frac{1}{2}-\frac{p}{2q}}}{p} + \frac{2^{\frac{1}{2}-\frac{p}{2q}}}{\frac{1}{2}-\frac{p}{2q}}.$$
By similar arguments, we can estimate that
	\begin{eqnarray}\label{ine3}
		J_2 &\leq& \int_0^1 (1-s)^{-\frac{1}{2}-\frac{p}{2q}}s^{-\frac{\alpha+\gamma}{2}} ds \sup_{t>0}t^{\frac{\alpha}{2}}\norm{u(t)}_{q,\infty,  {\lambda}}\sup_{t>0}t^{\frac{\gamma}{2}}\norm{\xi(t)}_{r,\infty,  {\lambda}} \cr 
		&\leq& C_2 \sup_{t>0} \norm{\begin{bmatrix}
				t^{\frac{\alpha}{2}}u(t)\\
				t^{\frac{\gamma}{2}}\theta(t)
		\end{bmatrix}}_{  {{\bf X}_{q,\lambda;r,\lambda}}} 
		\sup_{t>0} \norm{\begin{bmatrix}
				t^{\frac{\alpha}{2}}v(t)\\
				t^{\frac{\gamma}{2}}\xi(t)
		\end{bmatrix}}_{ {{\bf X}_{q,\lambda;r,\lambda}}},
	\end{eqnarray}
	where
	$$C_2 = \frac{2^{\frac{1}{2}-\frac{p}{2q}}}{\frac{p}{2q}+\frac{p}{2r}} + \frac{2^{\frac{1}{2}-\frac{p}{2r}}}{\frac{1}{2}-\frac{p}{2q}}.$$
	Combining \eqref{ine1}, \eqref{ine2}, \eqref{ine3} and the bilinear estimate \eqref{bestimate} in $H_{p,\infty}$, we obtain the bilinear estimate \eqref{BBestimate} in $H_{q,r,\infty}$.
\end{proof}

\begin{theorem}\label{stability}(Polynomial stability)
Let $n\geqslant 3$, $2<p<q \leq r <\infty$, $r>q'=\dfrac{q}{q-1}$, $p\leq  n$ and $b>\dfrac{p}{2}$ such that $\dfrac{1}{p}<\dfrac{1}{b}+\dfrac{1}{r}<\min\left\{ \dfrac{2}{p} + \dfrac{1}{q},\, 1\right\}$. The mild solution $\begin{bmatrix}
		\hat u\cr \hat \theta
	\end{bmatrix}$ obtained in Theorem \ref{wellposed} (where $\norm{\begin{bmatrix}
			\hat u\cr \hat \theta
	\end{bmatrix}}_{ { {\bf X}_{p,\lambda}}}$ is small enough) is polynomial stable in the sense that: if $\begin{bmatrix}
		v\cr \xi
	\end{bmatrix}$ is  { another} mild solution of the system \eqref{BouEq1} with the initial data $\begin{bmatrix}
		v_0\cr \xi_0
	\end{bmatrix}$ and the fields $g$ and $\omega$ such that $||| g |||_{\beta,b} = \sup_{t>0}{t^\beta}\norm{g(\cdot,t)}_{b,\infty,\lambda}$ and
	$|||g-\omega|||_{\beta,b}=\sup\limits_{t>0} t^{\beta}\norm{g(\cdot,t) - \omega(\cdot,t)}_{b,\infty,  {\lambda}}$ (where $\beta=1-\frac{p}{2b}$) and
	$\norm{\begin{bmatrix}\hat{u}_0-v_0\cr \hat{\theta}_0-\xi_0 \end{bmatrix}}_{   { {\bf X}_{p,\lambda}}}$ are small enough. Then, we have
	\begin{equation}\label{PolyStab}
		\sup_{t>0}\norm{\begin{bmatrix}
				t^{\frac{\alpha}{2}}(\hat{u}(\cdot,t)-v(\cdot,t))\cr t^{\frac{\gamma}{2}}(\hat{\theta}(\cdot,t)-\xi(\cdot,t))
		\end{bmatrix}}_{  {{\bf X}_{q,\lambda;r,\lambda}}} <C,
	\end{equation}
	where $\alpha = 1-\dfrac{p}{q}$, $\gamma = 1-\dfrac{p}{r}$ and $C$ is a positive constant.
\end{theorem}
\begin{proof}
Since $\begin{bmatrix}
		\hat u\cr \hat \theta
	\end{bmatrix}$ and $\begin{bmatrix}
		v\cr \xi
	\end{bmatrix}$ are solutions of the equation \eqref{mildsol},	we have $\begin{bmatrix}
		\hat u-v\cr \hat \theta-\xi
	\end{bmatrix}$ satisfies the following integral equation
	\begin{eqnarray}\label{StaEq}
		\begin{bmatrix}
			\hat u-v\\
			\hat \theta-\xi 
		\end{bmatrix}(t) &=& e^{-tL}\begin{bmatrix}
			\hat{u}_0-v_0\\
			\hat{\theta}_0-\xi_0
		\end{bmatrix} + B\left(\begin{bmatrix}
			\hat u\\
			\hat \theta
		\end{bmatrix}, \begin{bmatrix}
			\hat u\\
			\hat \theta
		\end{bmatrix} \right)(t) - B\left(\begin{bmatrix}
			v\\
			\xi
		\end{bmatrix}, \begin{bmatrix}
			v\\
			\xi
		\end{bmatrix} \right)(t) + T_g(\hat \theta)(t)- T_\omega(\xi)(t)\cr
		&=&e^{-tL}\begin{bmatrix}
			\hat u_0-v_0\\
			\hat \theta_0-\xi_0
		\end{bmatrix} + \int_0^t \nabla_x e^{-(t-s)L} \left( \begin{bmatrix}
			\mathbb{P}(\hat u\otimes \hat u)\\
			\hat u\hat \theta
		\end{bmatrix}(s) - \begin{bmatrix}
			\mathbb{P}(v\otimes v)\\
			v\xi
		\end{bmatrix}(s) \right)ds \cr
		&&+ T_g(\hat \theta)(t)-T_\omega(\xi)(t)\cr
		&=&e^{-tL}\begin{bmatrix}
			\hat{u}_0-v_0\\
			\hat{\theta}_0-\xi_0
		\end{bmatrix} + \int_0^t \nabla_x e^{-(t-s)L} \left( \begin{bmatrix}
			\mathbb{P}((\hat u-v)\otimes \hat u)\\
			(\hat u-v)\hat \theta
		\end{bmatrix}(s) + \begin{bmatrix}
			\mathbb{P}(v\otimes (\hat u-v))\\
			v(\hat \theta-\xi)
		\end{bmatrix}(s) \right)ds \cr
		&&+ T_g(\hat \theta-\xi)(t)+T_{g-\omega}(\xi)(t)\cr
		&=&e^{-tL}\begin{bmatrix}
			\hat{u}_0-v_0\\
			\hat{\theta}_0-\xi_0
		\end{bmatrix} + \int_0^t \nabla_x e^{-(t-s)L} \left( \begin{bmatrix}
			\mathbb{P}((\hat u-v)\otimes \hat u)\\
			(\hat u-v)\hat \theta
		\end{bmatrix}(s) + \begin{bmatrix}
			\mathbb{P}((v-\hat u)\otimes (\hat u-v))\\
			(v-\hat u)(\hat \theta-\xi)
		\end{bmatrix}(s) \right)ds \cr
		&&+ \int_0^t \nabla_x e^{-(t-s)L} \begin{bmatrix}
			\mathbb{P}(\hat u\otimes (\hat u-v))\\
			\hat u(\hat \theta-\xi)
		\end{bmatrix}(s) ds \cr
		&&+ T_g(\hat \theta-\xi)(t)+T_{g-\omega}(\xi - \hat\theta)(t) + T_{g-\omega}(\hat\theta)(t).
	\end{eqnarray}
	
	To prove the polynomial stability \eqref{PolyStab}, we show that integral equation \eqref{StaEq} has a mild solution in a small ball of $H_{q,r,\infty}$. First, we have
	\begin{eqnarray}
		\norm{\begin{bmatrix}
				t^{\frac{\alpha}{2}}(\hat u-v)\\ 
				t^{\frac{\gamma}{2}}(\hat\theta-\xi) 
			\end{bmatrix}(t)}_{  {{\bf X}_{q,\lambda;r,\lambda}}} &=& \norm{e^{-tL}\begin{bmatrix}
				t^{\frac{\alpha}{2}}(\hat{u}_0-v_0)\\
				t^{\frac{\gamma}{2}}(\hat{\theta}_0-\xi_0)
		\end{bmatrix}}_{  {{\bf X}_{q,\lambda;r,\lambda}}} \cr
		&&+ \int_0^t \norm{\nabla_x e^{-(t-s)L} \begin{bmatrix}
				t^{\frac{\alpha}{2}}\mathbb{P}((\hat u-v)\otimes \hat u)\\
				t^{\frac{\gamma}{2}}(\hat u-v)\theta
			\end{bmatrix}(s)}_{ {{\bf X}_{q,\lambda;r,\lambda}}} ds \cr
		&& + \int_0^t \norm{\nabla_x e^{-(t-s)L}\begin{bmatrix}
				t^{\frac{\alpha}{2}}\mathbb{P}((v-\hat u)\otimes (\hat u-v))\\
				t^{\frac{\gamma}{2}}(v-\hat u)(\theta-\xi)
			\end{bmatrix}(s)}_{  {{\bf X}_{q,\lambda;r,\lambda}}} ds \cr
		&& + \int_0^t \norm{\nabla_x e^{-(t-s)L} \begin{bmatrix}
				t^{\frac{\alpha}{2}}\mathbb{P}(\hat u\otimes (\hat u-v))\\
				t^{\frac{\gamma}{2}}\hat u(\hat \theta-\xi)
			\end{bmatrix}(s)}_{  {{\bf X}_{q,\lambda;r,\lambda}}} ds \cr
		&&+ \norm{t^{\frac{\alpha}{2}}T_g(\hat\theta-\xi)(t)}_{  {{\bf X}_{q,\lambda;r,\lambda}}} \cr
		&&+ \norm{t^{\frac{\alpha}{2}}T_{g-\omega}(\xi-\hat\theta)(t)}_{  {{\bf X}_{q,\lambda;r,\lambda}}}\cr
		&&+ \norm{t^{\frac{\alpha}{2}}T_{g-\omega}(\hat\theta)(t)}_{  {{\bf X}_{q,\lambda;r,\lambda}}}\cr
		&:=& I_0 + I_1 + I_2 + I_3 + I_4+I_5+I_6,
	\end{eqnarray}
	where
	\begin{eqnarray}
		&&I_0 = \norm{e^{-tL}\begin{bmatrix}
				t^{\frac{\alpha}{2}}(\hat{u}_0-v_0)\\
				t^{\frac{\gamma}{2}}(\hat{\theta}_0-\xi_0)
		\end{bmatrix}}_{  {{\bf X}_{q,\lambda;r,\lambda}}}\cr
		&& I_1 = \int_0^t \norm{\nabla_x e^{-(t-s)L} \begin{bmatrix}
				t^{\frac{\alpha}{2}}\mathbb{P}((\hat u-v)\otimes \hat u)\\
				t^{\frac{\gamma}{2}}(\hat u-v)\hat\theta
			\end{bmatrix}(s)}_{  {{\bf X}_{q,\lambda;r,\lambda}}}ds \cr
		&& I_2 = \int_0^t \norm{\nabla_x e^{-(t-s)L}\begin{bmatrix}
				t^{\frac{\alpha}{2}}\mathbb{P}((v-\hat u)\otimes (\hat u-v))\\
				t^{\frac{\gamma}{2}}(v-\hat u)(\hat\theta-\xi)
			\end{bmatrix}(s)}_{ {{\bf X}_{q,\lambda;r,\lambda}}} ds \cr
		&& I_3 = \int_0^t \norm{\nabla_x e^{-(t-s)L}\begin{bmatrix}
				t^{\frac{\alpha}{2}}\mathbb{P}(\hat u\otimes (\hat u-v))\\
				t^{\frac{\gamma}{2}}\hat u(\hat\theta-\xi)
		\end{bmatrix}(s)}_{  {{\bf X}_{q,\lambda;r,\lambda}}} ds \cr
		&& I_4 = \norm{t^{\frac{\alpha}{2}}T_g(\hat\theta-\xi)(t)}_{  {{\bf X}_{q,\lambda;r,\lambda}}}\cr
		&& I_5 = \norm{t^{\frac{\alpha}{2}}T_{g-\omega}(\xi-\hat\theta)(t)}_{  {{\bf X}_{q,\lambda;r,\lambda}}}\cr
		&& I_6 = \norm{t^{\frac{\alpha}{2}}T_{g-\omega}(\hat\theta)(t)}_{  {{\bf X}_{q,\lambda;r,\lambda}}}.
	\end{eqnarray}
	By using estimate \eqref{disp}, $I_0$ can be estimated as follows
	\begin{equation}\label{Ine1}
		I_0 \leq C\norm{\begin{bmatrix}
				\hat{u}_0-v_0\\
				\hat{\theta}_0-\xi_0
		\end{bmatrix}}_{  { {\bf X}_{p,\lambda}}}. 
	\end{equation}
	By using the bilinear estimate in $H_{q,r,\infty}$ we have that
	\begin{eqnarray}\label{Ine2}
		&&I_1, \, I_3 \leq K\norm{\begin{bmatrix}
				\hat u-v\\
				\hat\theta -\xi
		\end{bmatrix}}_{H_{q,r,\infty}}\norm{\begin{bmatrix}
				\hat u\\
				\hat\theta
		\end{bmatrix}}_{H_{q,r,\infty}}\cr
		&&I_2 \leq K\norm{\begin{bmatrix}
				\hat u-v\\
				\hat\theta -\xi
		\end{bmatrix}}^2_{H_{q,r,\infty}}
	\end{eqnarray}
	The terms $I_4$, $I_5$ and $I_6$ can be estimated as
	\begin{eqnarray}\label{Ine3}
		I_4 &\leq& t^{\frac{\alpha}{2}}\int_0^t (t-s)^{1-\frac{p}{2b}+\frac{\gamma}{2}-\frac{\alpha}{2}-1}s^{-1+\frac{p}{2b}-\frac{\gamma}{2}}s^{\frac{\gamma}{2}}\norm{(\hat\theta-\xi)(s)}_{r,\infty,  {\lambda}}ds \sup_{t>0}t^{1-\frac{p}{2b}}\norm{g(t)}_{b,\infty,  {\lambda}}\cr
		&\leq& \int_0^1 (t-s)^{-\frac{p}{2b}+\frac{\gamma-\alpha}{2}}s^{-1+\frac{p}{2b}-\frac{\gamma}{2}}(ts)^{\frac{\gamma}{2}}\norm{(\hat\theta-\xi)(ts)}_{r,\infty,  {\lambda}}ds \sup_{t>0}t^{1-\frac{p}{2b}}\norm{g(t)}_{b,\infty,  {\lambda}}\cr
		&\leq& L |||g|||_{\beta,b} \norm{\begin{bmatrix}
				\hat u-v\\
				\hat\theta -\xi
		\end{bmatrix}}_{H_{q,r,\infty}},
	\end{eqnarray}
	\begin{eqnarray}\label{Ine4}
		I_5 &\leq& t^{\frac{\alpha}{2}}\int_0^t (t-s)^{1-\frac{p}{2b}+\frac{\gamma}{2}-\frac{\alpha}{2}-1}s^{-1+\frac{p}{2b}-\frac{\gamma}{2}}s^{1-\frac{p}{2b}}\norm{(g-\omega)(s)}_{b,\infty,  {\lambda}}ds \sup_{t>0}t^{\frac{\gamma}{2}}\norm{(\xi-\hat\theta)(t)}_{r,\infty,  {\lambda}}\cr
		&\leq& \int_0^1 (t-s)^{-\frac{p}{2b}+\frac{\gamma-\alpha}{2}}s^{-1+\frac{p}{2b}-\frac{\gamma}{2}}(ts)^{1-\frac{p}{2b}}\norm{(g-\omega)(ts)}_{b,\infty,  {\lambda}}ds \sup_{t>0}t^{\frac{\gamma}{2}}\norm{(\xi-\hat\theta)(t)}_{r,\infty,  {\lambda}}\cr
		&\leq& M \sup_{t>0}t^{1-\frac{p}{2b}}\norm{(g-\omega)(t)}_{b,\infty,  {\lambda}} \norm{\begin{bmatrix}
				\hat u-v\\
				\hat\theta -\xi
		\end{bmatrix}}_{H_{q,r,\infty}}\cr
	&\leq& M |||g-\omega|||_{\beta,b} \norm{\begin{bmatrix}
				\hat u-v\\
				\hat\theta -\xi
		\end{bmatrix}}_{H_{q,r,\infty}}
	\end{eqnarray}
	and
	\begin{eqnarray}\label{Ine5}
		I_6 &\leq& t^{\frac{\alpha}{2}}\int_0^t (t-s)^{1-\frac{p}{2b}+\frac{\gamma}{2}-\frac{\alpha}{2}-1}s^{-1+\frac{p}{2b}-\frac{\gamma}{2}}s^{1-\frac{p}{2b}}\norm{(g-\omega)(s)}_{b,\infty,  {\lambda}}ds \sup_{t>0}t^{\frac{\gamma}{2}}\norm{\hat\theta(t)}_{r,\infty,  {\lambda}}\cr
	&\leq& \int_0^1 (t-s)^{-\frac{p}{2b}+\frac{\gamma-\alpha}{2}}s^{-1+\frac{p}{2b}-\frac{\gamma}{2}}(ts)^{1-\frac{p}{2b}}\norm{(g-\omega)(ts)}_{b,\infty,  {\lambda}}ds \sup_{t>0}t^{\frac{\gamma}{2}}\norm{\hat\theta(t)}_{r,\infty,  {\lambda}}\cr
		&\leq& M \sup_{t>0}t^{1-\frac{p}{2b}}\norm{(g-\omega)(t)}_{b,\infty,  {\lambda}} \norm{\begin{bmatrix}
				\hat u\\
				\hat\theta
		\end{bmatrix}}_{H_{q,r,\infty}}\cr
		&\leq& M |||g-\omega|||_{\beta,b} \norm{\begin{bmatrix}
				\hat u\\
				\hat\theta
		\end{bmatrix}}_{H_{q,r,\infty}}.
	\end{eqnarray}
	Here, we notice that the constants $L$ and $M$ are finite since the fact that $2<p<q \leq r <\infty$, $r>q':=\dfrac{q}{q-1}$ and $p\leq \min \left\{ n,\, 2b\right\}$, $\dfrac{1}{p}<\dfrac{1}{b}+\dfrac{1}{r}<\min\left\{ \dfrac{2}{p} + \dfrac{1}{q},\, 1\right\}$.
	
	Now we consider the map: $\Phi: H_{q,r,\infty} \to H_{q,r,\infty}$ given by
	\begin{eqnarray}
		\Phi \begin{bmatrix}U\\
			\Theta
		\end{bmatrix}&:=& e^{-tL}\begin{bmatrix}
			\hat{u}_0 - v_0\\
			\hat{\theta}_0 -\xi_0
		\end{bmatrix} + \int_0^t \nabla_x e^{-(t-s)L} \left( \begin{bmatrix}
			\mathbb{P}(U\otimes \hat u)\\
			U\hat \theta
		\end{bmatrix}(s) - \begin{bmatrix}
			\mathbb{P}(U\otimes U)\\
			U\Theta
		\end{bmatrix}(s) \right)ds \cr
		&&+ \int_0^t \nabla_x e^{-(t-s)L} \begin{bmatrix}
			\mathbb{P}(\hat u\otimes U)\\
			\hat u\Theta
		\end{bmatrix}(s) ds \cr
		&&+ T_g(\Theta)(t) - T_{g-\omega}(\Theta)(t) + T_{g-\omega}(\hat\theta)(t).
	\end{eqnarray}
	
	Using estimates \eqref{Ine1}, \eqref{Ine2}, \eqref{Ine3}, \eqref{Ine4} and \eqref{Ine5}, we see that for all $\begin{bmatrix}U\\
		\Theta
	\end{bmatrix}$ belongs to the ball with radius $\rho$ and centered at $\begin{bmatrix}0\\
		0
	\end{bmatrix}$, $\mathbb{B}_\rho \subset H_{q,r,\infty}$, then
	\begin{eqnarray}
		\norm{\Phi\begin{bmatrix}U\\
				\Theta
		\end{bmatrix}}_{H_{q,r,\infty}} &\leq& C\norm{\begin{bmatrix}\hat{u}_0-v_0\\
				\hat{\theta}_0-\xi_0
		\end{bmatrix}}_{  { {\bf X}_{p,\lambda}}} + 2K\norm{\begin{bmatrix}
				U\\
				\Theta
		\end{bmatrix}}_{H_{q,r,\infty}}\norm{\begin{bmatrix}
				\hat u\\
				\hat\theta
		\end{bmatrix}}_{H_{q,r,\infty}} + K\norm{\begin{bmatrix}
				U\\
				\Theta
		\end{bmatrix}}_{H_{q,r,\infty}}^2 \cr
		&&+ L\sup_{t>0}t^{1-\frac{p}{2b}}\norm{g(t)}_{b,\infty,  {\lambda}} \norm{\begin{bmatrix}
				U\\
				\Theta
		\end{bmatrix}}_{H_{q,r,\infty}}\cr
		&&+ M \sup_{t>0}t^{1-\frac{p}{2b}}\norm{(g-\omega)(t)}_{b,\infty,  {\lambda}} \norm{\begin{bmatrix}
				U\\
				\Theta
		\end{bmatrix}}_{H_{q,r,\infty}} \cr
		&&+ M \sup_{t>0}t^{1-\frac{p}{2b}}\norm{(g-\omega)(t)}_{b,\infty,  {\lambda}} \norm{\begin{bmatrix}
				\hat u\\
				\hat\theta
		\end{bmatrix}}_{H_{q,r,\infty}}\cr
		&\leq& C\norm{\begin{bmatrix}\hat{u}_0-v_0\\
				\hat{\theta}_0-\xi_0
		\end{bmatrix}}_{  { {\bf X}_{p,\lambda}}} + 2K\norm{\begin{bmatrix}
				U\\
				\Theta
		\end{bmatrix}}_{H_{q,r,\infty}}\norm{\begin{bmatrix}
				\hat u\\
				\hat\theta
		\end{bmatrix}}_{H_{q,r,\infty}} + K\norm{\begin{bmatrix}
				U\\
				\Theta
		\end{bmatrix}}_{H_{q,r,\infty}}^2 \cr
		&&+ L|||g|||_{\beta,b} \norm{\begin{bmatrix}
				U\\
				\Theta
		\end{bmatrix}}_{H_{q,r,\infty}} + M |||g-\omega|||_{\beta,b} \norm{\begin{bmatrix}
				U\\
				\Theta
		\end{bmatrix}}_{H_{q,r,\infty}} + M |||g-\omega|||_{\beta,b} \norm{\begin{bmatrix}
				\hat u\\
				\hat\theta
		\end{bmatrix}}_{H_{q,r,\infty}}
	\end{eqnarray}
	Hence, for $\rho$, $\norm{\begin{bmatrix}
			\hat u\\
			\hat\theta
	\end{bmatrix}}_{H_{q,r,\infty}}$, $\norm{\begin{bmatrix}\hat{u}_0-v_0\\
			\hat{\theta}_0-\xi_0
	\end{bmatrix}}_{  { {\bf X}_{p,\lambda}}}$, $|||g|||_{\beta,b}$ and $|||g-\omega|||_{\beta,b}$ are small enough, then $\norm{\Phi\begin{bmatrix}U\\
			\Theta
	\end{bmatrix}}_{H_{q,r,\infty}} \leq \rho$, i.e., $\Phi(\mathbb{B}_\rho)\subset \mathbb{B}_\rho$.
	
	Furthermore, for $\begin{bmatrix}U_1\\
		\Theta_1
	\end{bmatrix}, \,\begin{bmatrix}U_2\\
		\Theta_2
	\end{bmatrix} \in \mathbb{B}_\rho $ we have that 
	\begin{eqnarray}
		\norm{\Phi\begin{bmatrix}U_1-U_2\\
				\Theta_1-\Theta_2
		\end{bmatrix}}_{H_{q,r,\infty}} &\leq& 2K\norm{\begin{bmatrix}
				U_1-U_2\\
				\Theta_1-\Theta_2
		\end{bmatrix}}_{H_{q,r,\infty}}\norm{\begin{bmatrix}
				\hat u\\
				\hat\theta
		\end{bmatrix}}_{H_{q,r,\infty}} + K\norm{\begin{bmatrix}
				U_1-U_2\\
				\Theta_1-\Theta_2
		\end{bmatrix}}_{H_{q,r,\infty}}^2 \cr
		&&+ L|||g|||_{\beta,b} \norm{\begin{bmatrix}
				U_1-U_2\\
				\Theta_1-\Theta_2
		\end{bmatrix}}_{H_{q,r,\infty}} + M |||g-\omega|||_{\beta,b} \norm{\begin{bmatrix}
				U_1-U_2\\
				\Theta_1-\Theta_2
		\end{bmatrix}}_{H_{q,r,\infty}} \cr
		&\leq& P\norm{\begin{bmatrix}
				U_1-U_2\\
				\Theta_1-\Theta_2
		\end{bmatrix}}_{H_{q,r,\infty}},
	\end{eqnarray}
	where
	$$P=2K\norm{\begin{bmatrix}
			\hat u\\
			\hat\theta
	\end{bmatrix}}_{H_{q,r,\infty}} + K\rho + L|||g|||_{\beta,b} +  M |||g-\omega|||_{\beta,b}<1$$
provided by $\rho$, $\norm{\begin{bmatrix}
			\hat u\\
			\hat\theta
	\end{bmatrix}}_{H_{q,r,\infty}}$, $\norm{\begin{bmatrix}\hat{u}_0-v_0\\
			\hat{\theta}_0-\xi_0
	\end{bmatrix}}_{  { {\bf X}_{p,\lambda}}}$, $|||g|||_{\beta,b}$ and $|||g-\omega|||_{\beta,b}$ are small enough.
	
	Therefore, $\Phi$ is a contraction for such values of $\rho$, $\norm{\begin{bmatrix}
			\hat u\\
			\hat\theta
	\end{bmatrix}}_{H_{q,r,\infty}}$, $\norm{\begin{bmatrix}\hat{u}_0-v_0\\
			\hat{\theta}_0-\xi_0
	\end{bmatrix}}_{  { {\bf X}_{p,\lambda}}}$, $|||g|||_{\beta,b}$ and $|||g-\omega|||_{\beta,b}$. Hence, by the fixed point argument, 
	there exists a unique point $\begin{bmatrix}
		U\\
		\Theta
	\end{bmatrix}\in \mathbb{B}_\rho$. This point is also a unique solution of integral equation \eqref{StaEq} in $\mathbb{B}_\rho\in H_{q,r,\infty}$. This follows the polynomial stability \eqref{PolyStab}, and our proof is completed. 
\end{proof}

\begin{remark}
\begin{itemize}
\item[(i)]\label{Rem-Theo-1} (External force) The condition of  $g$
given in our main theorems (Theorem \ref{PeriodicLinearCasse}, Theorem \ref{wellposed} and Theorem \ref{stability}) 
{does not cover the case of gravitational field $g$ for specific numbers
$p$ and $b$. In fact, if we consider the gravitational field $g(x)=G\dfrac{x}{|x|^3}$, then $g\in L^{\frac{n}{2},\infty}(\mathbb{R}^n) = \mathcal{M}_{\frac{n}{2},\infty,0}(\mathbb{R}^n)\hookrightarrow \mathcal{M}_{\frac{p}{2},\infty,n-p}$ for $2<p\leq n$, where $G$ is the gravitational constant. We can choose $p=n$ and $b=\frac{n}{2}$, then we get
$$\limsup_{t\to +\infty}t^{1-\frac{p}{2b}}\norm{g}_{b,\infty,\lambda} = \norm{g}_{\frac{n}{2},\infty,0}\limsup_{t\to +\infty}t^0 \simeq G,$$
which can not small enough as we desire in the proofs of Theorems \ref{PeriodicLinearCasse}, \ref{wellposed} and \ref{stability}. However, these theorems are still valid if we assume that the coefficent of volume expansion $\kappa$ is
small enough. This condition guarantees the proofs of above theorems still hold since it is appeared the term $\kappa \limsup\limits_{t\to +\infty}t^{1-\frac{p}{2b}}\norm{g}_{b,\infty,\lambda}\simeq \kappa G$ in the proofs of above theorems (see the same in B\'enard problem in \cite[Corollary 2]{Fe2010}).
}
%also covers classes of time-independent $g. In fact, for $g(x,t)\equiv g(x)\in\mathcal{M}_{b_{\mu},\infty,\mu}$ with
%$\frac{n-\mu}{b_{\mu}}=\frac{n-\lambda}{b},$ $0\leq\mu\leq\lambda$ and $b\leq
%b_{\mu}$, the continuous inclusion $\mathcal{M}_{b_{\mu},\infty,\mu
%}\hookrightarrow\mathcal{M}_{b,\infty,\lambda}$ follows that
%\[ |||g|||_{\beta,b}=\limsup_{t\rightarrow0^{+}}t^{1-\frac{n-\lambda}{2b}%}\left\Vert g\right\Vert _{b,\infty,\lambda}\leq C\left\Vert g\right\Vert
%_{b_{\mu},\infty,\mu}\limsup_{t\rightarrow0^{+}}t^{1-\frac{n-\lambda}{2b}}=0.\]
%$For detail, we can consider arbitrary $g\in L^{q,\infty}(\mathbb{R}%^{n})=\mathcal{M}_{q,\infty,0}(\mathbb{R}^{n})\subset\mathcal{M}%_{b,\infty,\lambda}(\mathbb{R}^{n})$ where $q>b.$ Moreover, for an arbitrarily-large bounded set $A\subset\mathbb{R}^{n}$, we have $1_{A}\in L^{q,\infty}(\mathbb{R}^{n})$ and can take fields $g=$ $\tilde{g}\cdot1_{A}$ with $\tilde{g}\in L^{\infty}(\mathbb{R}^{n})$, such as a constant field $\tilde{g}$.

\item[(ii)] Since the weak-Morrey spaces are larger than weak-Lorentz and Morrey spaces, {our stability results extend the ones obtained in previous works \cite{Al2011,Fe2006,Fe2010,HuyXuan2022}. Indeed, the works \cite{Fe2006,HuyXuan2022} and \cite{Fe2010} established the stability for the Boussinesq  and stationary Boussinesq system in the framework of weak-$L^p$ spaces. Both of these works obtained the global well-posedness of small mild solutions for Boussinesq system in $BC(\mathbb{R}_+,\,L^{n,\infty}_\sigma(\Omega)\times L^{n,\infty}(\Omega))$ (where $\Omega$ is unbounded domains of $\mathbb{R}^n$ or whole space $\mathbb{R}^n$ with $n\geqslant 3$). Moreover, the solutions obtained in \cite{Fe2006,Fe2010,HuyXuan2022} are stable in $BC(\mathbb{R}_+,\,L^{r,\infty}_\sigma(\Omega)\times L^{r,\infty}(\Omega))$  with decay $t^{-\frac{1}{2}\left(1-\frac{n}{r}\right)}$ for $t>0$ and $r>n$. This is a special case of our stability result obtained in Theorem \ref{stability}. Indeed, in Theorem \ref{stability} we can choose $p=n$ and $q=r$, then we get $\alpha=\gamma = \frac{1}{2}\left( 1 - \frac{n}{r}\right)$. Therefore, we obtain the same decay $t^{-\frac{1}{2}\left( 1 - \frac{n}{r}\right)}$ of stable results in \cite{Fe2006,Fe2010,HuyXuan2022}. \\
On the other hand, the work \cite{Al2011} established the stability for Boussinesq system in the framework of Morrey space. In particular,
the author in \cite{Al2011} obtained the global well-posedness in $BC(\mathbb{R}_+, \mathcal{M}^\sigma_{p,n-p}(\mathbb{R}^n)\times \mathcal{M}_{p,n-p}(\mathbb{R}^n))$ and then proved the stability of the obtained solutions in $BC(\mathbb{R}_+, \, \mathcal{M}^\sigma_{q,n-p}(\mathbb{R}^n)\times \mathcal{M}_{r,n-p}(\mathbb{R}^n))$ 
with the stability decays of solutions which are exactly our decay \eqref{PolyStab} obtained in Theorem \ref{stability} (for more details, see Theorem 4.1 in \cite{Al2011}). Since we have the strict inclusion $\mathcal{M}_{p,n-p}(\mathbb{R}^n)\hookrightarrow \mathcal{M}_{p,\infty,n-p}(\mathbb{R}^n)$, our stability result can be considered as an extension of \cite{Al2011} in the sense of extended phase spaces.
}
\end{itemize}
\end{remark}

\subsection{Periodic solutions of Navier-Stokes equations revisited}
If we consider the zero-temperature case, i.e., $\theta=0$, then Boussinesq system
(\ref{BouEq}) becomes Navier-Stokes equations.
\begin{equation}\label{NSE}
\left\{
\begin{array}
[c]{rll}%
u_{t}-\Delta u+\mathbb{P}\operatorname{div}(u\otimes u)\!\! & = \mathbb{P}\dive F\quad & x\in\mathbb{R}^{n},\,t>0,\hfill\\
\operatorname{div}u\!\! & =\;0\quad & x\in\mathbb{R}^{n},\,t\geqslant 0,\\
u(x,0)\!\! & =\;u_{0}(x)\quad & x\in\mathbb{R}^{n}
\end{array}
\right.
\end{equation}
The existence and stability of periodic mild solutions for Navier-Stokes equations in weak-Lorentz spaces $L^{n,\infty}(\mathbb{R}^n)$ were established in some previous works (see for example \cite{Huy2014,Ya2000} and references therein). Applying Theorem \ref{wellposed} and Theorem \ref{stability}, we can extend these results on weak-Morrey spaces $\mathcal{M}_{p,\infty,n-p}$ which are larger than weak-Lorentz spaces $L^{n,\infty}(\mathbb{R}^n)$.
Our theorem is stated as follows:
\begin{theorem}\label{NavierStokes}
Let $n\geqslant 3$ and $2<p\leq n$. Setting $\lambda=n-p$ and assume that the external force $F
	\in BC(\mathbb R_+; \mathcal{M}^\sigma_{\frac{p}{2},\infty,\lambda})$ 
  is a $T$-periodic function with respect to the time. 
\begin{itemize}
\item[(i)]  The well-posedness of periodic mild solutions of Navier-Stokes equations is stated as: if  {the norm $\norm{F}_{BC(\r_+,\cal M^\sigma_ {\frac{p}{2},\infty,\lambda})}$ is small enough, then there exists a unique initial data $\hat u_0\in \cal M^\sigma_ {p,\infty,\lambda}$ which guarantees that for this initial data the Navier-Stokes equation \eqref{NSE} has a unique periodic mild solution $\hat u$ with the periodicity $T$} in a small ball of weak-Morrey space $\mathcal{M}^\sigma_{p,\infty,\lambda}$.

\item[(ii)] Consider $q$ such that $2<p<q<\infty$. The periodic mild solution $\hat u$ obtained in Assertion (i) is polynomial stable in the sense that: if $v$ is    {another} bounded mild solution of the Cauchy problem of Navier-Stokes equations\eqref{NSE}
with the initial $v_0$ such that $\norm{\hat{u}_0-v_0}_{p,\infty,\lambda}$ is small enough. Then, we have
\begin{equation*} 
\sup_{t>0}\norm{t^{\frac{\alpha}{2}}(\hat{u}(\cdot,t)-v(\cdot,t))}_{q,\infty,\lambda} <C,
\end{equation*}
where $\alpha = 1-\dfrac{p}{q}$ and $C$ is a positive constant.
\end{itemize}
\end{theorem}
\subsection*{Acknowledgment}
 This research is funded by Thuyloi University under Grant Number TLU.STF.23-08.


\begin{thebibliography}{99}                                    
	
\bibitem{Adams1971} {D. R. Adams,\textit{ Traces of potentials arising from translation invariant operators}, Ann. Scuola Norm. Sup. Pisa (3) 25 (1971), 203–217.}

\bibitem {Adams}D. R. Adams, \textit{Morrey spaces, Lecture Notes in Applied and Numerical Harmonic Analysis}, Birkh\"{a}user/Springer, Cham, 2015.

                                                           
\bibitem {Al2011}M.F. de Almeida and L.C.F. Ferreira, \textit{On the well-posedness and large time behavior for Boussinesq equations in Morrey spaces}, Differential integral equations 24 (7/8) (2011), 719-742.

\bibitem{Al2013}  M. F. de Almeida and L. C. F. Ferreira, {\it On the Navier-Stokes equations in the half-space with initial and boundary rough data in Morrey spaces}, J. Differential Equations 254 (2013), no. 3, 1548–1570

\bibitem{Al2022} M.F. de Almeida and L.S.M Lima, {\it Adams’ trace principle in Morrey–Lorentz spaces on $\beta$-Hausdorff dimensional surfaces}, Annales Fennici Mathematici, 46(2), 1161–1177.

 {\bibitem{Al2021}M.F. de Almeida, F. Marcelo ; Lima, S. M. Lidiane,  \textit{Adams’ trace principle in Morrey-Lorentz spaces on $\beta$-Hausdorff dimensional surfaces}, Ann. Fenn. Math. 46 (2021), no.2, 1161–1177.}

\bibitem{BeLo} J. Bergh and J. L\"ofstr\"om, {\it Interpolation spaces. An introduction}, Springer-Verlag, Berlin, 1976.

\bibitem {Br2012}L. Brandolese and M.E. Schonbek, \textit{Large time decay and
growth for solutions of a viscous Boussinesq system}, Trans. Amer. Math. Soc.
\textbf{364} (10) (2012), 5057-5090.

\bibitem {Br2020}L. Brandolese and J. He, \textit{Uniqueness theorems for the
Boussinesq system}, Tohoku Math. J. 72 (2) (2020), 283-297.

\bibitem {Ca1980}J.R. Cannon and E. DiBenedetto, \textit{The initial value
problem for the Boussinesq equations with data in $L^{p}$}, Approximation
Methods for Navier-Stokes Problems, Edited by Rautmann, R., Lect. Notes in
Math., Springer-Verlag, Berlin, 771 (1980), 129-144.

\bibitem {Chandra}S. Chandrasekhar, \textit{Hydrodynamic and Hydromagnetic
Stability}, Dover, New York, 1981.

\bibitem {Danchin2009}R. Danchin and M. Paicu, \textit{Global well-posedness
issue for the inviscid Boussinesq system with Youdovitch's type data}, Commun.
Math. Phys. 290 (2009), 1-14.

\bibitem {Danchin2008}R. Danchin and M. Paicu, \textit{Existence and
uniqueness results for the Boussinesq system with data in Lorentz spaces},
Phys. D 237 (10-12) (2008), 1444-1460.

\bibitem {Fe2006}L.C.F. Ferreira and E.J. Villamizar-Roa,
\textit{Well-posedness and asymptotic behaviour for the convection problem in
$\mathbb{R}^{n}$}, Nonlinearity 19 (9) (2006), 2169-2191.

\bibitem {Fe2010}L.C.F. Ferreira and E.J. Villamizar-Roa, \textit{On the
stability problem for the Boussinesq equations in weak-$L^{p}$ spaces},
Commun. Pure Appl. Anal. 9 (3) (2010), 667-684.

\bibitem {Fe2016}L.C.F. Ferreira, \textit{On a bilinear estimate in
weak-Morrey spaces and uniqueness for Navier-Stokes equations}, J. Math. Pures
Appl. 105 (2) (2016), 228-247.

\bibitem {FeXuan2023} L.C.F. Ferreira and P.T. Xuan, {\it On uniqueness of mild solutions for Boussinesq equations in Morrey-type spaces}, Applied Mathematics Letters, Vol. 137, num. 2 (2023), pages 1-6. 

\bibitem {Fi1969}P.C. Fife and D.D. Joseph, \textit{Existence of convective
solutions of the generalized Bernard problem which are analytic in their
norm}, Arch. Rational Mech. Anal. 33 (1969), 116-138.

\bibitem{GaSho} G. Galdi and H. Sohr, {\it Existence and uniqueness of time-periodic physically reasonable Navier–Stokes flow past a body}, Arch. Ration. Mech. Anal. 172 (2004) 363-406.

\bibitem {Graf2004}L. Grafakos, \textit{Classical and modern Fourier
analysis}, Pearson Education, Inc., Upper Saddle River, NJ, 2004.

\bibitem{Hatano20} N. Hatano, {\it Fractional operators on Morrey-Lorentz spaces and the Olsen inequality}, Math. Notes 107 (2020), no. 1–2, 63–79.

\bibitem{Hatano22} N. Hatano, {\it Atomic decomposition for Morrey-Lorentz spaces}, preprint (2022) 	arXiv:2212.13717 

\bibitem {Hi1995}T. Hishida, \textit{Global Existence and Exponential
Stability of Convection}, J. Math. Anal. Appl. 196 (2) (1995), 699-721.

\bibitem {Hi1997}T. Hishida, \textit{On a class of Stable Steady flow to the
Exterior Convection Problem}, Journal of Differential Equations 141 (1)
(1997), 54-85.

\bibitem {Huy2014} N. T. Huy, {\it Periodic motions of Stokes and Navier-Stokes flows around a rotating obstacle}, Arch Ration Mech Anal, 2014, 213: 689–703

\bibitem{HuyXuan2022} N.T. Huy, P.T. Xuan, L.T. Sac and V.T.N. Ha, {\it Existence and stability of periodic and almost periodic solutions to the Boussinesq in unbounded domains}, Acta Mathematica Scientia, Vol. 42, Iss. 5 (2022), pages 1875-1901

\bibitem {Ka1984}T. Kato, \textit{Strong $L^{p}$ solutions of the
Navier-Stokes equations in $R^{m}$ with applications to weak solutions}, Math.
Z. 187 (4) (1984), 471-480.

\bibitem{KoNa1994} H. Kozono and M. Yamazaki, {\it Semilinear heat equations and the Navier-Stokes equation with
distributions in new function spaces as initial data}, Comm. Partial Differential Equations 19 (1994), no. 5-6, 959–1014

\bibitem {Komo2015}C. Komo, \textit{Uniqueness criteria and strong solutions
of the Boussinesq equations in completely general domains}, Z. Anal. Anwend.
34 (2) (2015), 147-164.

\bibitem {Le2002} P.G. Lemarie-Rieusset, \textit{Recent Developments in the
Navier-Stokes Problem}, Research Notes in Mathematics, Vol. 431. Chapman \&
Hall/CRC, Boca Raton, 2002.

\bibitem {Li-Wang2021} Z. Li and W. Wang, \textit{Norm inflation for the
Boussinesq system}, Discrete Contin. Dyn. Syst. Ser. B 26 (10) (2021), 5449-5463.

\bibitem {Liu2014} X. Liu and Y. Li, \textit{On the stability of global
solutions to the 3D Boussinesq system}, Nonlinear Analysis 95 (2014), 580-591.

%\bibitem{Lu2003} G. Lukaszewicz, E.E. Ortega-Torres and M.A. Rojas-Medar, {\it Strong periodic solutions for a class of abstract evolution equations}, Nonlinear Analysis {\bf 54}, 6, 1045-1056 (2003).

\bibitem {Meyer1997} Y. Meyer, \textit{Wavelets, paraproducts, and
Navier-Stokes equations}, Current developments in mathematics, 1996
(Cambridge, MA), 105--212, Int. Press, Boston, MA, 1997.

\bibitem {Mo1991}H. Morimoto, \textit{Non-stationary Boussinesq equations},
Proc. Japan Acad. Ser. A math. Sci. 67 (5) (1991), 159-161.

\bibitem {Na2020}K. Nakao, \textit{On time-periodic solutions to the
Boussinesq equations in exterior domains}, J. Math. Anal. Appl. 482 (2) (2020), 123537, 16 pp.

\bibitem {Ya2000}M. Yamazaki, \textit{The Navier-Stokes equations in the weak-$L^{n}$ space with time-dependent external force}, Math. Ann. 317 (4) (2000), 635-675.

\bibitem{Phuc} {N. C. Phuc, \textit{Morrey global bounds and quasilinear Riccati type equations below the natural exponent}, J. Math. Pures Appl. (9) 102 (2014), no.1, 99–123.}

\bibitem{Ragusa} {M. A. Ragusa, \textit{Embeddings for Morrey-Lorentz spaces}, J. Optim. Theory Appl. 154 (2012),	no.2, 491–499}

\bibitem{SaFaHa2020}  {Y. Sawano, G. D. Fazio, D. I. Hakim, \textit{	Morrey Spaces: Introduction and Applications to Integral Operators and PDE’s, Volumes I $\&$ II}, New York, 1st Edition (2020), https://doi.org/10.1201/9781003042341}

\bibitem{Vi2010} Villamizar-Roa E J, Rodriguez-Bellido M A, Rojas-medar M A., {\it Periodic solutions in unbounded domains for the Boussinesq system}, Acta Mathematica Sinica, English Series, 2010, {\bf 26} (5): 837–862

\end{thebibliography}
\end{document}